    \documentclass[11pt]{article}
    \usepackage{amsmath,amsthm,amscd,amssymb,latexsym,verbatim, xfrac}
    \usepackage{amssymb}

\usepackage[normalem]{ulem}

\usepackage{graphicx,epsf,cite}
    
\usepackage[shortlabels]{enumitem}
    
\usepackage{xcolor}
\usepackage{graphicx}

    \usepackage[margin=1.2in]{geometry}

    \newtheorem{theorem}{Theorem}[section]
    
    \newtheorem{proposition}[theorem]{Proposition}
    \newtheorem{lemma}[theorem]{Lemma}
    \newtheorem{corollary}[theorem]{Corollary}
        \newtheorem{assumption}[theorem]{Assumption}
        \newtheorem{definition}[theorem]{Definition}
    \theoremstyle{definition}

    \newtheorem{remark}[theorem]{Remark}

    \newcounter{smalllist}

    \allowdisplaybreaks
    \numberwithin{equation}{section}
    
    \newcommand{\lb}{\label}

    \newcommand{\ol}{\overline}

    \newcommand{\beq}{\begin{equation}}
    \newcommand{\eeq}{\end{equation}}
    
    \newcommand{\bal}{\begin{align}}
    \newcommand{\eal}{\end{align}}
    \newcommand{\bals}{\begin{align*}}
    \newcommand{\eals}{\end{align*}}
    
    \newcommand{\R}{\ensuremath{\mathbb{R}}}
    
    \newcommand{\bbN}{{\mathbb{N}}}
    \newcommand{\bbR}{{\mathbb{R}}}
    
    \newcommand{\bbP}{{\mathbb{P}}}
    \newcommand{\bbE}{{\mathbb{E}}}

    \newcommand{\calB}{{\mathcal B}}

    \newcommand{\calF}{{\mathcal F}}

      \newcommand{\calO}{{\mathcal O}}
    
    \newcommand{\ep}{\varepsilon}
    \newcommand{\eps}{\varepsilon}

\date{\today}
\title{A Hele-Shaw problem with  interior and  free boundary oscillation: well-posedness and homogenization}

\author{Olga Turanova\thanks{Department of Mathematics, Michigan State University, Wells Hall, 619 Red Cedar Road,
East Lansing, MI 48824, \texttt{turanova@msu.edu}.}, Yuming Paul Zhang\thanks{Department of Mathematics and Statistics, Auburn University, Parker Hall, 221 Roosevelt Concourse, Auburn, AL 36849, \texttt{yzhangpaul@auburn.edu}.}
}

\begin{document}
     \maketitle
\begin{abstract} 
We investigate a Hele-Shaw type free boundary problem in one spatial dimension, where heterogeneities appear both on the free boundary and within the interior of the positivity set. Our contributions are twofold. First, we establish well-posedness and a comparison principle for the problem by introducing a novel notion of viscosity flows. Second, under the assumption that the coefficients are stationary ergodic, we prove a stochastic homogenization result. Our results are new even in the periodic setting. To derive the effective free boundary velocity, we use a new approximation that accounts for both interior homogenization and free boundary propagation.

\end{abstract}

\section{Introduction} \lb{S1}

We consider the following free boundary problem posed in one spatial dimension:

\begin{equation}\lb{main}
\left\{
\begin{aligned}
    &-\partial_x(A(x) \partial_x p)=F(x) &&\quad\text{ in }\{p>0\},\\
  & \partial_t p=V(x,  \partial_x p)| \partial_x p| &&\quad\text{ on }\partial\{p>0\}.
\end{aligned}   
\right.
\end{equation}
Here, $p=p(x,t)$ is unknown, and $A$, $F$, and $V$ are given functions. 
The  condition $\partial_t p=V(x,\partial_x p)| \partial_x p|$ in \eqref{main} says that the free boundary, $\partial \{p>0\}$, expands with velocity
\[
V(x,\partial_x p)\quad\text{ at }x\in \partial \{p(\cdot,t)>0\}.
\]
We assume  that $A$ and $F$ are strictly positive,  $A$ is $C^{1+\sigma}$, and $V=V(x,q)$ is positive and   monotonically increasing in $q$ when $q>0$ and decreasing in $q$ when $q<0$ (see Assumption \ref{assump2} for the precise hypotheses on the coefficients).

When $A\equiv 1$, $F\equiv 0$, and $V(x,q)=|q|$, the equation \eqref{main} becomes the classical Hele-Shaw problem, which was introduced by 
H. S. H. Shaw in 1898  to describe the dynamics of an incompressible viscous Newtonian fluid between two parallel, narrowly separated surfaces \cite{hele1898experiments}.  More recently, \eqref{main} has been used in the modeling of tumor cells \cite{PQV,David_S,jacobs2022tumor,SulakTuranova}, which we discuss more in Section \ref{ss1.3}, and in the context of congested population dynamics \cite{maury2010,CKY}. The Hele-Shaw problem can also be transformed into the one-phase Muskat problem with constant gravity, see \cite{Hongjie,schwab2024well}. 

Even in the case of constant coefficients ($A\equiv 1$, $F\equiv 0$, and $V(x,q)=|q|$), solutions to  \eqref{main} can develop singularities in finite time; thus, weak solutions are needed. Our focus is  on viscosity solutions, which were introduced for \eqref{main} in the constant coefficient case in \cite{kim2003}. 
In our first main result, Theorem \ref{T.1}, we prove existence and uniqueness of viscosity solutions for \eqref{main} in one spatial dimension. To this end, we introduce a novel notion of viscosity flows and we establish the comparison principle for viscosity flows\footnote{A notion of viscosity sub- and super-flows was introduced by Barles and Souganidis in \cite{barles1998new} to describe the motion of fronts that depends on their curvature. Although we use similar terminology, the setting in \cite{barles1998new} is completely different from the Hele-Shaw type flow that we consider here.}.

Next, we aim to understand the impact of   fluctuations in the medium that occur at scale $\ep$ for $\ep>0$ very small. This is modeled by the equation
\begin{equation}\lb{1.1}
\left\{
\begin{aligned}
  & -\partial_x(A(x,\eps^{-1}x,\omega) \partial_x p_\eps)=F(x,\eps^{-1}x,\omega) &&\quad\text{ in }\{p_\eps>0\},\\
&   \partial_t p_\eps=B(x,\eps^{-1}x,\omega)| \partial_x p_\eps|^2+G(x,\eps^{-1}x,\omega)| \partial_x p_\eps| &&\quad\text{ on }\partial\{p_\eps>0\},
\end{aligned}   
\right.
\end{equation}
where $\omega$ is an element of a probability space $(\Sigma, \mathcal{F}, P)$,  the coefficients 
 $A$, $B$, $F$ are strictly positive, and $G$ is nonnegative. The coefficients are assumed to be regular enough and  stationary ergodic  with respect to the fast variable $\eps^{-1}x$ (see Assumption \ref{assumption} and Definition \ref{D.ergodic}). Heuristically, the latter assumption captures  that, on the microscopic scale $\ep$, the coefficients vary  randomly  but in a self-averaging way. 
Due to this, it is expected that, as $\ep\rightarrow 0$, solutions of \eqref{1.1} should converge to those of the macroscopic problem
\begin{equation}\lb{1.3}
\left\{
\begin{aligned}
&    -\partial_x(\overline A(x) \partial_x \overline p)=\overline F(x) &&\quad\text{ in }\{\overline p>0\},\\
&   \partial_t \overline p=\overline V(x,\partial_x \overline p)| \partial_x \overline p| &&\quad\text{ on }\partial\{\overline  p>0\},
\end{aligned}   
\right.
\end{equation}
where the homogenized coefficients $\ol A$, $\ol V$, $\ol F$ are averages (in the correct sense) of the original ones. 
This is exactly what we establish in our second main result, Theorem \ref{thm:main}.

To prove Theorem \ref{thm:main}, we follow a general approach that appears throughout the  literature on stochastic homogenization of elliptic equations: applying the subadditive ergodic theorem to a suitable quantity to reveal the effective behavior in the $\ep\rightarrow 0$ limit.  However, the challenge in studying \eqref{1.1} is that homogenization occurs 
both in the interior region, due to the oscillating elliptic operator, and at the free boundary. These two processes
are fundamentally different in nature --- they occur in different dimensions --- and finding an appropriate subadditive quantity that accounts for their interaction forms the most novel part of our work.

In the remainder of the introduction we state our main results and the necessary assumptions (Section \ref{ss:results}), describe the main ideas behind the proofs (Section \ref{ss:ideas}), provide a brief literature review (Section \ref{ss:lit}), and describe the connections between models of tissue growth and the free boundary problem under study here (Section \ref{ss1.3}).

\subsection{Assumptions and main results} \label{ss:results}
Before stating our main results, we introduce some notation. 
For a space-time set $\Omega\subseteq \R^2$, we use $\Omega(t)$ to denote the time slice of $\Omega$ at time $t$; namely, 
\[
\Omega(t) = \{x \mid (x,t)\in \Omega\}.
\]
For a nonnegative lower semicontinuous function $p:\bbR\times [0,\infty)\to [0,\infty)$, we denote its positivity set and the  boundary of its positivity set by,
\[
\Omega_p:=\{(x,t)\in\bbR\times [0,\infty) \mid p(x,t)>0\}\quad\text{and}\quad\Gamma_p:=\{(x,t)\in\bbR\times [0,\infty) \mid x\in\Gamma_p(t)\}.
\]

\subsubsection{Existence and uniqueness}
We now state the assumptions needed for our result on existence and uniqueness for \eqref{main}.
\begin{assumption}[Coefficients of \eqref{main}]
\label{assump2} We assume that the following hold for
\[
A, F: \R\rightarrow \R \quad \text{and}\quad V:\R\times \R\rightarrow \R.
\]
\begin{enumerate}[(i)]
    \item \label{item:assump1.1.1} $A$, $F$ are  uniformly positive, Lipschitz continuous and bounded, and $A_x$ is uniformly $\sigma$-H\"{o}lder continuous for some $\sigma\in (0,1)$. 

    \item \label{item:assump1.1 2}
    There exist $C>0$ and a continuous function $\eta:(0,\infty)\to(0,\infty)$ such that, for all $x_1,x_2,q_1,q_2\in\bbR$, we have,
    \beq\lb{cond4}
|V(x_1,q_1)-V(x_2,q_2)|\leq C(1+|q_1|)|x_1-x_2|+C|q_1-q_2|.
\eeq
Furthermore, for all $\gamma\in (0,1)$ and for all $x_1,x_2,q_1\in\bbR$ with $C|x_1-x_2|\leq \gamma$, we have,
\begin{align}
\lb{cond5'}
V(x_1,(1+\gamma)q_1)&\geq V(x_1,q_1)\geq \eta(|q_1|);
\\
\lb{cond5}
(1+\gamma)^2 V(x_1,q_1)+C|x_1-x_2|^2/\gamma&\geq V(x_2,(1+\gamma)q_1).
\end{align}
\end{enumerate}
\end{assumption}

Our first main result is:
\begin{theorem}[Well-posedness]
\label{T.1}
Under Assumption \ref{assump2}, let $\Omega_0\subseteq\R$ be open and bounded. Then there exists a unique viscosity solution to \eqref{main} with initial data having support $\Omega_0$.
\end{theorem}
The definition of viscosity solutions for \eqref{main} is stated in Definition \ref{def23}. 

\begin{remark}
    
Assumption \ref{assump2} \ref{item:assump1.1 2} implies that $V(\cdot, q)$ is strictly positive for each $q$, which guarantees that the support of $\ol p$ immediately expands at $t=0$. (We remark  that this property --- immediate expansion of the support of the solution --- also played a role in \cite{kim2003,KimMellet2009}.) 
This, as well as the monotonicity condition \eqref{cond5'} on $V$, is used to obtain the uniqueness of solutions. The condition \eqref{cond5} is technical.

\end{remark}

\subsubsection{Homogenization}
In order to state our second main result, we need the following definition. 
\begin{definition}[Stationary ergodic]
\label{D.ergodic}
Let $(\Sigma,\calF,P)$ be a probability space, and let 
\[
\left\{\tau_y \mid \tau_y:\Sigma\to\Sigma,\text{ with }y\in\bbR\right\}
\]
be a  measure-preserving transformation group such that $\tau_y\circ \tau_{z}=\tau_{y+z}$.
\begin{enumerate}[(i)]
    \item 
The transformation group $\{\tau_y\}$ is {\it ergodic} if the following holds: $\tau_yE=E$ for some $E\in\calF$ and all $y\in\bbR$ implies $\bbP[E]\in\{0,1\}$.
\item We say that 
a random variable $f:\bbR\times\Sigma\to\bbR$ is stationary if
\[
f(y+z,\omega)=f(y,\tau_z\omega) \quad \text{for all }y,z\in\bbR\text{ and }\omega\in\Sigma.
\]   
\end{enumerate}
\end{definition}

Next, we make precise our assumptions on the coefficients of \eqref{1.1}. 
\begin{assumption}[Coefficients of \eqref{1.1}]
\label{assumption}
We assume that the following hold for
\[
A, B, F, G:\bbR\times\bbR\times\Sigma\to\bbR. 
\]
\begin{enumerate}[(i)]
    \item \label{item:assump A F}$A$ and $F$ are  uniformly positive, continuous and bounded; $A_x(x,y,\omega)$ and $A_y(x,y,\omega)$ are uniformly bounded and $\sigma$-H\"{o}lder continuous in $x$ and $y$ for some $\sigma\in(0,1)$.

    \item $B$ is uniformly positive, bounded, and Lipschitz continuous in $x$ and $y$; $G$ is nonnegative and uniformly bounded, and $\sqrt{G(x,y,\omega)}$ is uniformly Lipschitz continuous in $x$ and $y$.

    \item For each fixed $x$, $A(x,\cdot,\cdot)$, $B(x,\cdot,\cdot)$,  $F(x,\cdot,\cdot)$ and $G(x,\cdot,\cdot)$ satisfy the stationary, ergodic assumption. 

    \item Either $G$ is uniformly strictly positive on its domain or $G\equiv 0$.
\end{enumerate}

\end{assumption}
In Lemma \ref{lem:assumption ep}, we show that if the coefficients $A$, $B$, $F$, and $G$ in equation \eqref{1.1} satisfy Assumption \ref{assumption}, then, for each fixed $\omega\in \Sigma$ and $\ep>0$, the assumptions of Theorem \ref{main} are satisfied, and therefore \eqref{1.1} has a unique viscosity solution. 

We are now ready to state:

\begin{theorem}[Homogenization]
\label{thm:main}
Given a probability space $(\Sigma,\calF,\bbP)$, suppose Assumption \ref{assumption} holds.
Then there exist deterministic functions $\ol A(x)$, $\ol V(x, q)$, 
and $\ol F(x)$, which satisfy Assumption \ref{assump2}, 
such that the following holds. For any $\calO\subseteq\bbR$ open and bounded,  there exists $\Sigma_0\subseteq\Sigma$ of full measure such that, for any $\omega\in\Sigma_0$ and almost every $t>0$, we have
\[
p_\eps(\cdot,t,\omega) \to \ol p(\cdot,t)
\]
locally uniformly, where  $p_\eps(\cdot,\cdot,\omega)$ is the solution to \eqref{1.1} with $\Omega_{\ol p_\eps(\cdot, \cdot, \omega)}(0)=\calO$  and $\ol p$ is the unique viscosity solution to \eqref{1.3} 
with $\Omega_{\ol p}(0)=\calO$.
\end{theorem}

\begin{remark}
\begin{enumerate}[(i)]
    \item We use the assumption that $G$ is either  strictly positive or identically 0 to establish the uniqueness of the limiting problem in Theorem \ref{thm:main}. We believe this assumption is technical.

\item The effective coefficients $\ol A(x)$
and $\ol F(x)$ are explicit: see \eqref{eq:barA barF}.

\item Our result on stochastic homogenization implies that the  corresponding periodic homogenization result holds as well, which is  new for Hele-Shaw type problems with both interior and free boundary oscillation.

\end{enumerate}

\end{remark}

\subsection{Strategy of proof}
\label{ss:ideas}
In this subsection we describe the main ideas behind the proofs of our two results,  Theorem \ref{T.1} and Theorem \ref{thm:main}.

\subsubsection{Existence and uniqueness}

The main idea behind our proof of the comparison principle, which is central to how we establish well-posedness for \eqref{main}, is as follows. Suppose that $p_1$ is a subsolution and $p_2$ is a supersolution of \eqref{main} with $\Omega_1(0)\subsetneq \Omega_2(0)$, where
\[
\Omega_i:=\left\{p_i>0\right\}\quad\text{and}\quad \Omega_i(t)=\{p_i(\cdot, t)>0\}.
\]
The goal is to show $\Omega_1(t)\subsetneq\Omega_2(t)$ for all $t\geq 0$.

To show that this holds, we 
assume for contradiction that $\overline\Omega_1\not\subseteq\Omega_2$. Thus, there is a first time $t_0>0$ for which $ \partial\Omega_{1}(t_0)\cap \partial\Omega_{2}(t_0)$ is nonempty.  Let $x_0$ be one attaching free boundary point, i.e. $x_0\in \partial\Omega_{1}(t_0)\cap \partial\Omega_{2}(t_0)$.
On the one hand, since $t_0$ is the time at which  $\Omega_1$ ``caught up" to $\Omega_2$, the free boundary speed of $\Omega_1$ should be no smaller than the one of $\Omega_2$. Consequently,
\beq\lb{V1}
V(x_0,\partial_xp_2(x_0,t_0))\leq V(x_0,\partial_xp_1(x_0,t_0)).
\eeq
On the other hand, since $\Omega_{1}(\cdot)\subseteq\Omega_{2}(\cdot)$ up to time $t_0$, we have $p_1\leq p_2$ up to time $t_0$. Thus, 
\[
|\partial_xp_1(x_0,t_0)|\leq |\partial_xp_2(x_0,t_0)|.
\]
However, by the monotonicity assumption on $V$, we get
\beq\lb{V2}
V(x_0,\partial_xp_2(x_0,t_0))\geq V(x_0,\partial_xp_1(x_0,t_0)) .
\eeq
If one of the inequalities \eqref{V1} and \eqref{V2} were  strict, we would obtain a contradiction. 

This strategy was introduced by Kim in \cite{kim2003} in the setting of the homogeneous Hele-Shaw problem. In order to obtain the desired contradiction, the argument in \cite{kim2003} was carried out not for the sub- and supersolutions themselves, but for their sup- and inf-convolutions. These  regularizations are used throughout the viscosity solutions literature, including in Caffarelli and Salsa's work on regularity in free boundary problems \cite{cbook}. 
In the case when the interior operator is the Laplacian (such as in \cite{kim2003}), the sup‐convolution of a subsolution remains a subsolution (and similarly for inf‐convolution). However, the inhomogeneous elliptic operator in \eqref{main} prevents this property from holding: convolution perturbs the operator’s coefficients and breaks the sub/supersolution structure. Thus, the method of \cite{kim2003} does not carry over directly to our setting.

To overcome this, we focus on the domains $\Omega_i$ rather than the functions $p_i$. Namely, we introduce a new notion of viscosity sub‐ and super-flows: given some space-time region $\Omega$, we solve for the exact solutions to the elliptic equation in its interior at each time ---   the fact that we are in one spatial dimension allows us to avoid regularity issues --- and formulate the free boundary condition in terms of test functions (see Definitions \ref{D.1.2}, \ref{D.1.3}, and \ref{D.1.4}). We establish a comparison principle for sub- and super-flows (Theorem \ref{L.cp}) and use Perron's method to obtain well-posedness (Theorem \ref{T.3.4}).  We also define viscosity solutions to \eqref{main} --- see Definitions \ref{D.1.2'}, \ref{D.1.3'}, and \ref{def23} --- and deduce Theorem \ref{T.1} from our well-posedness result for viscosity flows.

\subsubsection{Homogenization}
\label{sss:intro hom}
The main challenge in proving our result on stochastic homogenization, Theorem \ref{thm:main}, is handling the combination of interior and free boundary oscillation. The most high-level outline of our approach  is similar to that of previous works:  applying the subadditive ergodic theorem to extract the average of an appropriate quantity. However, as we will describe in the literature review, previously-developed techniques don't apply here due to the interior oscillation and the lack of compactness of the random setting. Thus, the  novelty in our work is in how we find the correct subadditive quantity and how we use it to show that homogenization occurs.  

Finding the appropriate subaddive quantity is closely tied to  identifying the effective velocity $\overline V(x_0, q)$ for $x_0, q\in \R$. We will now give the hueristics behind this crucial step. We  consider \eqref{1.1} with  coefficients ``frozen" at $x_0$,
\beq
\label{eq:frozen}
a(x,\omega):=A(x_0,x,\omega),\quad b(y,\omega):=B(x_0,y,\omega)\quad \text{and}\quad g(y,\omega):=G(x_0,y,\omega), 
\eeq
and 
 study solutions $p$ with fixed slope $q$ at $-\infty$ and positivity set $\Omega_p(t)=(-\infty, S(t))$, where $S(0)=x_0$:
\begin{equation}
\label{eq:motivating pde}
\left\{
\begin{aligned}
 &   -\partial_x(a(x,\omega) \partial_x p)=0 &&\text{ in }\{p>0\}=\cup_{t\geq 0}(-\infty,S(t))\times\{t\},\\
&   \frac{d}{dt}S(t) =\frac{\partial_t p}{|\partial_x p|}=b(x,\omega)| \partial_x p|+g(x,\omega)  &&\text{ on }\partial\{p>0\},\\
&\text{$p$ has a slope of $q$ as $x$ becomes negative.}
\end{aligned}   
\right.
\end{equation}
We want to understand the behavior of the free boundary location $S(t)$. However, we cannot work with \eqref{eq:motivating pde} directly, since we do not know whether it is well-posed. Instead, we consider an analogous problem on a finite  domain $(x_0-N, S(t))$, for $N\in [0,\infty)$ and derive an effective ODE for $S(t)$ as $N\rightarrow -\infty$. Namely, we  let $v(\cdot,\omega)=v(\cdot,t,\omega)$ be the solution to
\begin{equation}\lb{newBD}
\left\{
\begin{aligned}
   &-\partial_x(a(x,\omega) \partial_x v(x,\omega))=0\qquad\qquad \text{ for } x\in (x_0-N,S(t)),\\
    &v(x_0-N,\omega)= q(S(t)-x_0+N),\qquad v(S(t),\omega)=0.
\end{aligned}   
\right.
\end{equation}
The boundary conditions ensure that the solution $v(\cdot, t)$ has macroscopic slope $q$ for each $t>0$. 
 A direct computation yields 
an explicit expression for $v_x(S(t),\omega)$: 
\begin{align}
v_x(S(t),\omega)&=-q\left(\frac{1}{S(t)-x_0+N}\int_{x_0-N}^{S(t)}\frac{1}{a(y,\omega)}dy\right)^{-1}\frac{1}{a(S(t),\omega)}.\lb{newBD2} 
\end{align}
The Birkhoff--Khinchin Theorem implies that the quantity in the parenthesis on the righthand side of \eqref{newBD2} converges to a deterministic quantity: more precisely, 
\beq\lb{new2.10}
\lim_{N\to \infty}\frac{1}{S(t)-x_0+N}\int_{x_0-N}^{S(t)}\frac{1}{a(y,\omega)}dy=\frac1{\ol a},
\eeq
where $\ol a$ is a constant that is independent of  $S(t)$ and $\omega$. 
(Moreover, in Lemma \ref{L.2.2} we establish that this convergence is uniform in the appropriate sense,  which will be crucial in the proof of existence of the deterministic homogenized boundary velocity.) 
Thus, for $N$ large, we expect,
\[
v_x(S(t),\omega) \approx -q\frac{\bar a}{a(S(t),\omega)}.
\]
Therefore, heuristically, the free boundary condition 
yields,
\[
\frac{d}{dt} S(t)=b(S(t),\omega)| \partial_x v(S(t),t, \omega)|+g(S(t),\omega)\approx q \frac{\bar a }{a(S(t),\omega)}b\left(S(t, \omega)\right)+g(S(t),\omega).
\]
Thus, we expect that  the effective location of the free boundary point  should obey
\[
\frac{d}{dt}S(t, \omega) = q\frac{\bar a }{a(S(t),\omega)}b\left(S(t, \omega)\right)+ g\left(S(t, \omega)\right),\quad S(0)=x_0.
\]
This is the effective ODE that we are seeking. 
We prove that the arrival time associated with this ODE is subadditive  and apply the subadditive ergodic theorem to obtain the homogenized arrival time (Lemma \ref{L.4.5}), which we then use to define the effective velocity $\overline V (x_0, q)$ (Corollary \ref{P.2.12}). We remark that in this argument we are relying on the fact that the problem is set in one spatial dimension.

Once $\overline V$ is defined, we have to prove that solutions $p_\ep$ of \eqref{1.1} converge to those of \eqref{1.3}. To this end, we   show that the supports $\Omega_{p_*}$ and $\Omega_{p^*}$  of
\[
p_{* }(x,t,\omega):=\liminf_{\eps\to 0,y\to x,s\to t}p_{\eps}(y,s,\omega) \quad \text{ and }\quad p^{* }(x,t,\omega):=\limsup_{\eps\to 0,y\to x,s\to t}p_{\eps}(y,s,\omega)
\]
are, respectively,  a superflow and a subflow of the effective equation \eqref{1.3}.  
To verify that $\Omega_{p_*}$ satisfies the definition of viscosity superflow (the argument for $\Omega_{p^*}$ is analogous), suppose that $(x_0,t_0)$ is a right-hand side free boundary point and $\varphi(x,t)$ is a smooth function that touches $p_* $ from below at $(x_0,t_0)$ locally for $t\leq t_0$. Let $q_0:=-\partial_x\varphi(x_0,t_0)>0$ and $r_0:=\partial_t\varphi(x_0,t_0)/q_0$. Assume for contradiction that
\beq\lb{intocon}
{r_0}<\ol V(x_0,-q_0).
\eeq
Since $\varphi$ is below $p_*$, we get $ q_0\leq |(p_*)_x|(x_0,t_0)$. Also, the free boundary of $\varphi$ moves no slower than that of $p_*$ and so
\beq\lb{fasterspeed}
{r_0}\geq (p_*)_t(x_0,t_0)/{|(p_*)_x|}(x_0,t_0).
\eeq
However, using the definition of $\overline V$ and its properties, we find,
\beq\lb{effective}
\text{the average of }\frac{(p_\eps)_t}{|(p_\eps)_x|}(x_0,t_0)\approx \ol V(x_0,-(p_*)_x).
\eeq
It follows that $\frac{(p_*)_t}{|(p_*)_x|}(x_0,t_0)= \ol V(x_0,-(p_*)_x)$. Then, by \eqref{intocon} and  \eqref{fasterspeed}, we obtain
\[
\ol V(x_0,-q_0)>r_0\geq \frac{(p_*)_t}{|(p_*)_x|}(x_0,t_0)= \ol V(x_0,-(p_*)_x),
\]
which contradicts with $q_0\leq |(p_*)_x|(x_0,t_0)$ due to the monotonicity  of $\ol V$.

Once we know that $\Omega_{p^*}$ and $\Omega_{p_*}$ are a subflow and superflow, the comparison principle for flows implies $\Omega_{p^*}\subseteq \Omega_{p_*}$. Using this, together with the fact that $\Omega_{p_*}\subseteq \Omega_{p^*}$ (which holds by definition),  allows us to conclude that $\Omega_{p^*}= \Omega_{p_*}$ is a viscosity flow, as desired. The rigorous proof is presented in Section \ref{S.5}.

\subsection{Literature review}
\label{ss:lit}

The study of well-posedness of the classical Hele-Shaw flow, that is, equation \eqref{1.1} with $A=B\equiv 1$, $V(x,q)=q$, and $G=F\equiv 0$ in general spatial dimension, was begun by  Elliot and Janovsky \cite{EJ}, who showed the well-posedness of weak solutions in $H^1$.
Escher and Simonett \cite{escher1997classical} obtained short-time existence results for classical solutions. Later, Kim \cite{kim2003} proved the well-posedness of viscosity solutions. More recently, Schwab, Tu, and Turanova \cite{schwab2024well} obtained the well-posedness of the Hele-Shaw problem with gravity (otherwise known as the one-phase Muskat problem) in all spatial dimensions in the graph setting. We refer the reader to these works for further literature on well-posedness.

As for regularity, Choi, Jerison, and Kim \cite{CJK,choi2009local} showed that if the initial free boundary for the Hele-Shaw flow is Lipschitz with small Lipschitz constant then the free boundary becomes instantly smooth. Dong, Gancedo, and Nguyen \cite{Hongjie,dong23} obtained global-in-time regularity for the one-phase Muskat problem in the graph setting in spatial dimension $2$ and $3$. 
Figalli, Ros-Oton, and Serra  \cite{figalli2020generic}  used the connection of the Hele-Shaw flow to the obstacle problem to establish a generic regularity result (i.e. an estimate that shows that the singular set of the free boundary is small) in spatial dimension 2. 
More recently, Kim and Zhang \cite{kim2024regularity} proved that, for Hele-Shaw flows with source and drift terms, flat free boundaries are actually Lipschitz.

Periodic homogenization for \eqref{1.1} with $A\equiv 1$ and $F\equiv 0$ in general spatial dimension has been studied by Kim in \cite{homofb,kim2012homogenization,kim2009error}. The case $G\equiv 0$ and $B$  strictly positive was addressed in \cite{homofb}.  The works  \cite{kim2012homogenization,kim2009error} addressed homogenization in the presence  of potentially negative boundary velocity and established the rate of convergence.  In contrast with our setting,   \cite{homofb,kim2012homogenization,kim2009error} concerned problems with oscillation  on the free boundary but not in the interior of the positivity set.
More importantly,  the periodicity in \cite{homofb,kim2012homogenization,kim2009error} provided an important compactness property that is absent in the random setting that we consider, necessitating quite different techniques. This is also one of the main reason that our result  only holds for spatial dimension one.

The most closely-related existing work on stochastic homogenization is \cite{KimMellet2009} by Kim and Mellet, which addressed the homogenization of the Hele-Shaw problem in the case $A\equiv 1$ and $F\equiv 0$ and in general spatial dimension, with both periodic and random coefficients. Their strategy relied on transforming the Hele-Shaw flow into an obstacle problem and applying homogenization techniques for the obstacle problem. This approach does not apply in our setting because the transformation breaks down in the presence of inhomogeneity within the positivity set.  
We also mention that there is a vast literature on   stochastic homogenization  for elliptic equations, beginning with the work \cite{PV1981} of Papanicolaou and Varadhan; see \cite{armstrong2014regularity,bella2017stochastic,gu2017high} and the references therein.

\subsection{Applications and perspectives: models of tissue growth}\lb{ss1.3}
The problem \eqref{1.1} is closely related to the   inhomogeneous porous medium equation,
\begin{align} \label{eqn: PME}
         \partial _t u_k = \nabla \cdot \left( \alpha(x)u_k \nabla p_k \right)  + u_k(1-p_k) \text{ on }\R^d\times (0,\infty),
\end{align}
where the pressure $p_k$ is given in terms of the density $u_k$ by the constitutive law,
\begin{equation*}
p_k = \frac{k}{k-1} \left( \frac{u_k }{\beta(x)} \right) ^ {k-1}, \quad k\geq 2.
\end{equation*}
This equation models the spread of cancer tumor in a heterogeneous environment. The term $u_k(1-p_k)$ represents that the tumor's growth is constrained by the pressure within the tissue, and the coefficients $\alpha$ and $\beta$ represent heterogeneity in the underlying medium and in the cellular packing density, respectively. 
In \cite{SulakTuranova}, it is established that  as $k\rightarrow+\infty$ (called the ``incompressible" or ``stiff pressure" limit), the pressure $p_k$  converges to a solution  of a weak formulation of a free boundary problem of Hele-Shaw type,
\[
\left\{
\begin{aligned}
    &-\nabla\cdot\left(\alpha(x)\beta(x)\nabla p_\infty\right) =  \beta(x) (1-p_\infty)  \quad &&\text{ in }\{p_\infty >0\},\\
    &\partial_tp_\infty=\alpha(x)|\nabla p_\infty|^2  \quad &&\text{ on }\partial\{p_\infty >0\}.
\end{aligned}
\right.
\]
The link between the porous medium equation and the Hele-Shaw problem via the incompressible limit  was first established in  \cite{CaffarelliFriedman1987, GQ, PQV} and has since been extended to a variety of contexts.  
This connection, in the case of \eqref{eqn: PME}, inspired our study of \eqref{1.1}. Furthermore,  stochastic homogenization for a porous medium-type equation was established in \cite{patrizi2022stochastic}. We leave for future work the question of possible interactions between homogenization and the incompressible limit.

\subsection{Outline}
Section \ref{ss:viscosity} is devoted to the definitions of and lemmas about viscosity solutions and viscosity flows. In Section \ref{ss:comparison}, we prove the comparison principle for viscosity flows, from which well-posedness (Theorem \ref{T.1}) follows as a corollary. Section \ref{S.random} concerns the random environment and presents preliminary lemmas for stochastic homogenization. In Section \ref{S.4}, we introduce the effective equation and determine the effective velocity. Finally, Section \ref{S.5} is devoted to the proof our second main result, Theorem \ref{thm:main}, on stochastic homogenization.

\subsection{Acknowledgements}
 O. Turanova acknowledges support from NSF grant DMS-2204722. Y. P. Zhang acknowledges support from NSF CAREER grant DMS-2440215 and Simons Foundation Travel Support MPS-TSM-00007305.

\section{Viscosity flows and viscosity solutions}\label{ss:viscosity}
We will now define viscosity solutions (Section \ref{ss:visc soln def}) and viscosity flows (Section \ref{s.2.2}). Then, in Section \ref{ss:regular property}, we define the notion of regular boundary points. We  establish that at regular points of sub- and super-flows, the appropriate inequality holds in the free boundary condition   (Proposition \ref{P.3.9}). 
The definitions, as well as  the comparison, existence, and uniqueness results of the next section, will apply both to the microscopic scale equation,
\eqref{1.1}, for each fixed $\omega\in \Sigma$, as well as to the macroscopic problem \eqref{1.3}. 
Unlike \cite{kim2003}, our approach is to define viscosity flows, which give rise to viscosity solutions.  The target in our definition is  space-time sets instead of functions.

\subsection{Viscosity solutions}
\label{ss:visc soln def}
We will now make precise the notion of viscosity solution to the generic problem \eqref{main}.
Let $D\subset\R$ be open and let $T>0$.

\begin{definition}[Viscosity subsolution]\lb{D.1.2'}
A non-negative upper semicontinuous function $p$ defined in $D\times(0,T)$ is a viscosity subsolution of \eqref{main} if for every $\phi\in C^{2,1}_{x,t}(D\times (0,T))$ such that $p-\phi$ has a local maximum in $\left(\overline{\Omega_p}\cap\{t\leq t_0\}\right)\cap\left(D\times (0,T)\right)$ at $(x_0,t_0)$, we have:
\begin{itemize}
    \item If $(x_0,t_0)\in \Omega_p$, then $
    -(A \phi_x)_x(x_0,t_0)\leq F(x_0)$ holds.
    \item If $(x_0,t_0)\in\Gamma_p$, then either 
    \begin{align*}
    |\phi_x(x_0,t_0)|&= 0, \quad\text{or}\\
\phi_t(x_0,t_0)-V(x_0,\phi_x(x_0,t_0))|\phi_x(x_0,t_0)|&\leq 0.
        \end{align*}
\end{itemize}
\end{definition}

\begin{definition}[Viscosity supersolution]\lb{D.1.3'}
A non-negative lower semicontinuous function $p$ defined in $D\times (0,T)$ is a {\it viscosity supersolution} of \eqref{main} if for every $\phi\in C^{2,1}_{x,t}(D\times (0,T))$ such that $p-\phi$ has a local minimum in $D\times\{t\leq t_0\}$ at $(x_0,t_0)$, then
\begin{itemize}
    \item If $(x_0,t_0)\in \Omega_p$, then $
    -(A \phi_x)_x(x_0,t_0)\geq F(x_0)$ holds.
    \item If $(x_0,t_0)\in\Gamma_p$, then either 
    \begin{align*}
    |\phi_x(x_0,t_0)|&= 0, \quad\text{or}\\
\phi_t(x_0,t_0)-V(x_0,\phi_x(x_0,t_0))|\phi_x(x_0,t_0)|&\geq 0.
        \end{align*}
\end{itemize}
\end{definition}

\begin{remark}
In the literature, viscosity supersolutions are often defined so that if $(x_0,t_0)\in\Gamma_p$, then either 
\beq\lb{olddef}
    \begin{aligned}
-(A\phi_x)_x(x_0,t_0)&\geq F(x_0), \quad \text{or}\\
    |\phi_x(x_0,t_0)|&= 0, \quad\text{or}\\
\phi_t(x_0,t_0)-V(x_0,\phi_x(x_0,t_0))|\phi_x(x_0,t_0)|&\geq 0.
        \end{aligned}
\eeq
Similarly, the condition $-(A\phi_x)_x(x_0,t_0)\leq F(x_0)$ appears in the definition of viscosity subsolution. We removed them in Definitions \ref{D.1.3} and \ref{D.1.3'}, because, as we will now demonstrate, they are actually unnecessary.  

Indeed, suppose that $p-\phi$ has a local minimum in $\{t\leq 0\}$ at $(0,0)\in\Gamma_p$ and \eqref{olddef} holds with $x_0=t_0=0$. 
Without loss of generality, we assume that $0$ is a right-hand side free boundary point of $\{p(0,\cdot)>0\}$.
Let us show that it is not possible that
\[
    -(A\phi_x)_x(0,0)\geq F(0),
\]
but, at the same time,
\[
    |\phi_x(0,0)|> 0\quad\text{and}\quad
\phi_t(0,0)-V(0,\phi_x(0,0))|\phi_x(0,0)|< 0.
\]
Suppose for contradiction that the above three inequalities hold. Since $0$ is a right-hand side free boundary, we have $\phi_x(0,0)<0$.
Let us take
\[
\tilde\phi(x,t):=\phi(x,t)+\eps t+\eps x+\eps^{-1}x^2.
\]
Since $A(0)>0$, $A\in C^1$, and $\phi$ is smooth,  we can take $\eps>0$ to be sufficiently small such that
\beq\lb{c1}
-(A\tilde\phi_x)_x(0,0)=-(A\phi_x)_x(0,0)-\eps A_x(0)-2\eps^{-1}A(0)< F(0),
\eeq
\beq\lb{c2}
   \tilde\phi_x(0,0)< 0\quad\text{and}\quad
\tilde\phi_t(0,0)-V(0,\tilde \phi_x(0,0))|\phi_x(0,0)|< 0.
\eeq

Next, the definition of $\tilde\phi$ implies  $\tilde\phi(x,t)\leq \phi(x,t)$ for all $t<0$ and for $-\eps^2<x<0$. Further, $\phi_x(0,0)<0$ implies that $\tilde \phi(x,t)\leq 0$ when $t<0$ and $x>0$ is sufficiently small. Putting this together with the fact that $p(\cdot,t)$ is supported inside $\{x<0\}$ when $t\leq 0$, we find that $p-\tilde \phi$ has a local minimum in $\{t\leq 0\}$ at $(0,0)\in\Gamma_p$. However, this, \eqref{c1}, and \eqref{c2} contradicts with the assumption that \eqref{olddef} holds. So, indeed, we can remove the condition $-(A\phi_x)_x(x_0,t_0)\leq F(x_0)$ from the definition of viscosity supersolution.

\end{remark}

\begin{remark}[Our definition vis-\`{a}-vis that of \cite{kim2003}]
\label{rem:def kim}
We note that, in the definition of viscosity sub- and super-solutions in   \cite{kim2003}, the condition ``$|\phi_x(x_0,t_0)|= 0$'' is excluded. However, a parallel argument as in \cite[Lemma 2.5]{kim2} yields that the requirement at the free boundary in Definitions \ref{D.1.2'} and \ref{D.1.3'} can be simplified for testing against functions with nonzero gradient. 
\end{remark}

For a function $p(x,t)$,
define the half-relaxed limits,
\beq\lb{half}
p^\# (x,t):=\limsup_{y\to x,\,s\to t}p(y,s)\quad\text{and}\quad p_{\# }(x,t):=\liminf_{y\to x,\,s\to t}p(y,s).
\eeq
Often, our solutions $p(\cdot,t)$ are given by solving elliptic equations in open subsets of $\bbR$, in which case they are uniformly H\"{o}lder continuous in space, and we have,
\[
p^\# (x,t)=\limsup_{s\to t}p(x,s)\quad\text{and}\quad p_{\# }(x,t)=\liminf_{s\to t}p(x,s).
\]

\begin{definition}[Viscosity solution]\lb{def23}
We say that a non-negative function $p$ is a {\it viscosity solution} of \eqref{main} if $p^\# $ is a viscosity subsolution of \eqref{1.1} and $p=p_\# $ is a viscosity supersolution of \eqref{main}.
\end{definition}

\subsection{Viscosity flows}\lb{s.2.2}

We start with a number of definitions. Recall that, for a space-time set $\Omega\subseteq \R^2$, we denote $\Omega(t)$ as the time slice at $t$. 

\begin{definition}[Admissible]\lb{D.ad}
We  call a set $\Omega$ {\it admissible} if it is bounded, and $\Omega(\cdot)$ is non-decreasing in Hausdorff distance, and $\Omega(t)$ is an open set for each $t$.
\end{definition}

\begin{definition}[Semicontinuity]
We say that $\Omega$ is lower semicontinuous if $\Omega(t)$ is lower semicontinuous in Hausdorff distance; we say that $\Omega$ is upper semicontinuous if $\Omega(t)$ is upper semicontinuous in Hausdorff distance.
\end{definition}

For any space time set $\Omega$, we define $\Omega_\#$ by 
\[
\Omega_\#(t):=\liminf_{s\to t}\Omega(s)=\bigcup_{c\to 0}\bigcap_{|s-t|<c}\Omega(s),
\]
and $\Omega^\#$ by 
\[
\Omega^\#(t):=\limsup_{s\to t}\Omega(s)=\bigcap_{c\to 0}\bigcup_{|s-t|<c}\Omega(s).
\]
Then $\Omega_\#$ is lower semicontinuous and  $\Omega^\#$ is upper semicontinuous.
\begin{definition} [Associated functions]  
We let $p(\cdot,t)$ be the unique solution to the following elliptic equation,
\begin{equation}
    \label{eq:associated}
-\partial_x(A(x) \partial_x p)=F(x)\quad\text{ in }\Omega(t)
\end{equation}
with $0$ boundary data. 
We say that $p$ is {\it associated} with the set $\Omega$. 
\end{definition}

It is easy to see that if $\Omega$ is lower or upper semicontinuous, then the corresponding function $p$ is also lower or upper semicontinuous.

In the case of space dimension $1$, since $\Omega(t)$ is bounded, we immediately have that $p(\cdot,t)$ is Lipschitz continuous for each $t$.

\begin{definition}[Viscosity subflow]\lb{D.1.2}
A space-time upper semicontinuous set $\Omega $ is said to be a viscosity subflow of \eqref{main} in $\R\times(0,T)$ if the following holds  with $p$ associated with $\Omega$. 
For every $\phi\in C^{2,1}_{x,t}(\R\times (0,T))$ such that $p-\phi$ has a local maximum in $\overline\Omega\times\{t\leq t_0\}$   
at $(x_0,t_0)\in \partial\Omega$, then  one of the following holds:
\begin{align*}
    |\phi_x(x_0,t_0)|&= 0, \quad\text{or}\\
\phi_t(x_0,t_0)-V(x_0,\phi_x(x_0,t_0))|\phi_x(x_0,t_0)|&\leq 0.
\end{align*}
\end{definition}

\begin{definition}[Viscosity superflow]\lb{D.1.3}
A space-time lower semicontinuous set $\Omega $ is said to be a viscosity superflow of \eqref{main} in $\R\times(0,T)$ if the following holds with $p$ associated with $\Omega$. 
For every $\phi\in C^{2,1}_{x,t}(\R\times (0,T))$ such that $p-\phi$ has a local minimum in $\R\times\{t\leq t_0\}$  
at $(x_0,t_0)\in \partial\Omega$, then  one of the following holds:
\begin{align*}
    |\phi_x(x_0,t_0)|&= 0, \quad\text{or}\\
\phi_t(x_0,t_0)-V(x_0,\phi_x(x_0,t_0))|\phi_x(x_0,t_0)|&\geq 0.
\end{align*}
\end{definition}

\begin{definition}[Viscosity flow]\lb{D.1.4}
We say that a space-time set $\Omega$ is a {\it viscosity flow} of \eqref{main} if it is a viscosity subflow and  $\Omega_\#$ a viscosity superflow of \eqref{main}.
\end{definition}

The following properties follow from the definitions.
\begin{lemma}\lb{L.def}
Given a viscosity subflow (resp. superflow) of \eqref{main}, its associated function is a viscosity subsolution (resp. supersolution). In particular, the associated function of a viscosity flow is a viscosity solution.

On the other hand, if $p$ is a viscosity solution to \eqref{main}, then its positivity set $\Omega$ satisfies that $\Omega^\#$ is a viscosity flow.
\end{lemma}
\begin{proof}
The claims mainly follow from the unraveling the definitions. 
Let us only prove the last statement. Because $p$ is a viscosity solution of \eqref{main}, for each $t$, it is a viscosity solution to the elliptic equation in its positivity set. So, $p$ is the associated function to its own positivity set. The conclusion then follows immediately from the definitions. 
\end{proof}

Because of the lemma, to find a viscosity solution to \eqref{main}, it suffices to look for a viscosity flow. 
However, we comment that the positivity set of a 
viscosity subsolution (resp. supersolution) might not be a viscosity subflow (resp. superflow).

\medskip

We will also need the following definition. 
\begin{definition}[Strictly separated]\lb{D.1.5}
Let $D\subset \R$. 
We say that a pair of functions $p_1,p_2:\overline{D}\to [0,\infty)$ are strictly separated (denoted by $p_1\prec p_2$) in $D$  if $\overline\Omega_{p_1}\subseteq {\rm Int }\,\Omega_{p_2}$, and $p_1(x)<p_2(x)$ in $\overline{\Omega_{p_1}}\cap\overline{D}$.

\end{definition}
This says that the supports of the two functions are separated and that, in the support of the smaller function, the two functions are strictly ordered.

\subsection{Properties of regular free boundary points}
\label{ss:regular property}
We now establish a property of sub and superflows that we'll use in our proof of comparison in the next section. 
\begin{definition}[Regular free boundary points]
We say that a boundary point $(x_0,t_0)\in\partial\Omega$ with $t_0\in (0,T)$ is {\it regular from the interior} if there exists a space-time ball $B\subseteq \Omega$ such that $\overline\Omega\cap \overline B=\{(x_0,t_0)\}$; we say it is {\it regular from the exterior} if there exists a space-time ball $B\subseteq \Omega^c$ such that $\overline\Omega\cap \overline B=\{(x_0,t_0)\}$.
\end{definition}

\begin{proposition}[Properties of regular boundary points] \lb{P.3.9}
Suppose that $\Omega$ is a viscosity subflow, and $(x_0,t_0)\in\partial\Omega$ is regular from the exterior, and the tangent line of the exterior ball is given by $ (t-t_0)-k(x-x_0)=0$ for some $k\geq 0$. Let $p$ be associated with $\Omega$. Then, $k>0$, and for $\alpha:=-p_x(x_0,t_0)$, 
\[
V(x_0,-\alpha)\geq 1/k.
\]

Suppose that $\Omega$ is a viscosity superflow, and $(x_0,t_0)\in\partial\Omega$ is regular from the interior, and the tangent line of the interior ball is given by $\beta (t-t_0)-(x-x_0)=0$ for some $\beta\geq 0$. Let $p$ be associated with $\Omega$. Then, 
for $\alpha:=-p_x(x_0,t_0)$, 
\[
V(x_0,-\alpha)\leq \beta.
\]
\end{proposition}
\begin{proof}
Let us only prove the first claim. After shifting and without loss of generality, we can assume   $(x_0,t_0)=(0,0)$, and 
$(-\delta,0)\subseteq \Omega(0)$ for some  $\delta>0$. Then $x_0=0$ is a right-hand side free boundary point and $\alpha>0$. 

Let $\eps\in (0,1)$ be arbitrary.
Since $t=kx$ is the tangent line to the attaching ball from exterior, we have 
\begin{equation}
\label{eq:p on RHS Sigma}
    p(x,t)=0 \text{ for all $(x,t)$ such that $t\leq (k+\eps)x$, $t\leq 0$, and $|x|,|t|$ sufficiently small.}
\end{equation} 
It is clear that \eqref{eq:p on RHS Sigma} holds with $k$ replaced by $k'$, so long as  $k'>0$ and $k'\geq k$. We fix one such $k'$.

Next, since $-p_x(0,0)=\alpha>0$, and $p$ is non-decreasing, we have, for all $l>0$ sufficiently small,
\begin{equation}
\label{eq:p(-l)}
p(-l,t)\leq (\alpha+\eps)l\quad\text{ for all }t\leq 0.
\end{equation}
For some $C_0>0$ and $l>0$ to be determined, define
\[
\phi(x,t):=(t-(k'+2\eps)x)\frac{(\alpha+2\eps)l}{t+(k'+2\eps)l}-C_0(t-(k'+2\eps)x)^2
\]
and 
\[
\Sigma_\eps:=\left\{(x,t)|\, t\in (-(k'+2\eps)l/2,0),\,x\in (-l,{t}/(k'+2\eps))\right\}.
\]

We will show that $p-\phi$ has a local maximum in $\Sigma_\eps$ at $(0,0)$.  
 To this end, we first note that, for $(x,t)\in\Sigma_\eps$, 
\begin{equation}
\label{eq:in Sigma ep}
0\leq t-(k'+2\eps)x\leq (k'+2\eps)l,\quad \frac{\alpha+2\eps}{k'+2\eps}\leq \frac{(\alpha+2\eps)l}{t+(k'+2\eps)l}\leq \frac{2(\alpha+2\eps)}{k'+2\eps}.
\end{equation}
In addition, we have, for $(x,t)\in \Sigma_\ep$,
\begin{align*}
\phi(x,t)&=(t-(k'+2\eps)x)\left[\frac{(\alpha+2\eps)l}{t+(k'+2\eps)l}-C_0(t-(k'+2\eps)x)\right]\\
&\geq (t-(k'+2\eps)x)\left[\frac{2(\alpha+2\eps)}{k'+2\eps}-C_0(k'+2\eps)l\right].
\end{align*}
Thus, by taking 
$l<\frac{2(\alpha+2\ep)}{C_0(k'+2\ep)^2}$, we ensure that 
$\phi(x,t)>0$ for all $(x,t)\in \Sigma_\ep$. 
Furthermore, for $t\in (-(k'+2\eps)l/2,0)$,
\beq\lb{9.21}
 \phi(-l,t)\geq (\alpha+\eps)l+l\eps-C_0((k'+2\eps)l)^2\geq (\alpha+\eps)l.
\eeq

Next, direct computation yields
\begin{align*}
\phi_x(x,t)&=-(k'+2\eps)\frac{(\alpha+2\eps)l}{(k'+2\eps)l+t}+2C_0(k'+2\eps)(t-(k'+2\eps)x),\\
\phi_{xx}(x,t)&=-2C_0(k'+2\eps)^2,\qquad \phi_t(0,0)=\frac{\alpha+2\eps}{k'+2\eps}.
\end{align*}
Moreover, for $(x,t)\in\Sigma_\eps$,
\begin{align*}
    -\partial_x &(A \partial_x \phi )-F = -A\phi_{xx} -A_x \phi_{x} -F\\
    &=  2C_0A(k'+2\eps)^2+A_x\frac{(k'+2\eps)(\alpha+2\eps)l}{(k'+2\eps)l+t}-2C_0A_x(k'+2\eps)(t-(k'+2\eps)x)-F.
    \end{align*}
    Using \eqref{eq:in Sigma ep} to bound the second and third terms from below, we find,
    \begin{align*}
    -\partial_x (A \partial_x \phi )-F
    &\geq 2C_0A(k'+2\eps)^2-|A_x|(\alpha+2\eps)-2C_0|A_x|(k'+2\eps)^2l-F\\
    &= 2C_0(k'+2\ep)^2(A-|A_x|l) - |A_x|(\alpha+2\ep)-F>0,
\end{align*}
where we ensure that the last inequality holds by first taking 
\[
C_0:=\max_x \left\{ \left[|A_x|(\alpha+2\ep)+F\right]/\left[(k'+2\ep)^2A\right]\right\}
\]
and then taking $l<\max_x\{A/(2|A_x|)\}$.  Thus, we have that $\phi$ is a supersolution of \eqref{eq:associated} on $(-l, {t}/(k'+2\eps))$ for each $t\in (-(k'+2\eps)l/2,0)$.

For each $t\in (-(k'+2\eps)l/2,0)$, consider the interval
$I(t):=(-l, {t}/(k'+2\eps))\cap \{x|\, p(x,t)>0\}$. Due to \eqref{eq:p on RHS Sigma} with $k$ replaced by $k'$, we have 
$I(t)\subsetneq (-l, {t}/(k'+2\eps))$ and, since $\phi$ is positive on $\Sigma_\ep$, we also find that $\phi>p$ holds on the right endpoint of $I(t)$.  Moreover,  \eqref{eq:p(-l)} and \eqref{9.21} imply $\phi(-l, t)\geq p(-l, t)$. Therefore, the comparison principle for \eqref{eq:associated} on $I(t)$ implies $\phi(x,t)\geq p(x,t)$ for $x\in I(t)$. Since this holds for each $t\in (-(k'+2\eps)l/2,0)$, we find $\phi\geq p$ holds on $\Sigma_\ep$. Since $\phi(0,0)=0=p(0,0)$, we see that $p-\phi$ has a local maximum on $\Sigma_\ep$ at $(0,0)$.

Thus, we can now apply the definition of viscosity subflow and compare $p$ with $\phi$ in $\Sigma_\eps$ to get at $(0,0)$,
\[
\phi_t-V(0,\phi_x)|\phi_x|\leq 0
\]
which implies 
\[
V(0,-(\alpha+2\eps))(\alpha+2\eps)\geq \frac{\alpha+2\eps}{k'+2\eps}.
\]
By continuity of $V$, passing $\eps\to 0$ yields the claim with $k'$ replaced by $k$. Since we only assumed $k'>0$ and $k'\geq k$ at the beginning, this implies that $1/k$ is finite and the claim holds for $k$ as well.
\end{proof}

\section{Comparison principle for viscosity flows}
\label{ss:comparison}
In this section, we establish the comparison principle for viscosity flows  and well-posedness of viscosity solutions  of \eqref{main}.

Similarly to the setting of free boundary problems in \cite{cbook}, we need to regularize our viscosity sub- and super- solutions to establish comparison. Section \ref{ss:reg flows} is devoted to  defining  the regularizations and establishing some of their important properties.   Then, in Section \ref{ss:comp and wp}, we use these to prove the comparison principle (Theorem \ref{L.cp}). Then we establish well-posedness for \eqref{main} via Perron's method (Theorem \ref{T.3.4}). Finally, in Section \ref{ss:wp ep}, we deduce that \eqref{1.1} is well-posed as well.

\subsection{Regularized flows}
\label{ss:reg flows}

We use the following notation for distance:
\[
|(x,t)-(y,s)|:=\sqrt{|x-y|^2+|t-s|^2}.
\]

Let two admissible sets $\Omega_1$ and $\Omega_2$  be, respectively, viscosity sub- and superflows to \eqref{main} 
in $\R\times (0,T)$. For some $r,\delta\in (0,1)$ and $h\in [0,1)$, we rescale the sets in time as follows:
\beq\lb{5.0}
\Omega_1^\delta :=\left\{(x,t)\mid  (x,(1+\delta)^{-2}t)\in\Omega_1\right\},\quad \Omega_2^\delta :=\left\{(x,t)\mid  (x,(1-\delta)^{-2}t)\in\Omega_2\right\}.
\eeq
Then, we define $U_1$ to be a neighborhood of $\Omega_1^\delta $,
\beq\lb{5.1}
U_1:=\left\{(x,t)\in \R\times (0,T)\, \big\vert \,  |(x,t)-(y,s)|<r- ht,\,(y,s)\in\Omega_1^\delta \right\},
\eeq
and, for  $h<r/T$, we define $U_2$ to be an interior subset of $\Omega_2^\delta $,  
\beq\lb{5.2}
U_2:=\left\{(x,t)\in \R\times (0,T) \, \big\vert \, |(x,t)-(y,s)|>r-ht,\,(y,s)\in (\Omega_2^\delta )^c\right\}.
\eeq

It is clear that 
\[
\Omega_1\subseteq U_1\quad\text{and}\quad U_2\subseteq \Omega_2.
\]
In Proposition \ref{P.4.2}, we'll show that the boundary points of $U_1$ are regular from the interior and the boundary points of $U_2$ are regular from the exterior.

Note that, by assumption, $\Omega_i(\cdot)$ is admissible and thus, by Definition \ref{D.ad}, it is non-decreasing.  So, by the definitions, any tangent line of $\partial U_1$ at a right-hand side free boundary point $(x_0,t_0)$ can be written as $\beta(t-t_0)-(x-x_0)=0$ with some $\beta\geq -{h}/{\sqrt{1-h^2}}\geq-h$, and the tangent line of $\partial U_2$  at a right-hand side free boundary point $(x_0,t_0)$ can be written as $(t-t_0)-k(x-x_0)=0$ with some $k\geq -{h}/{\sqrt{1-h^2}}\geq-h$.

In terms of functions, let $p_1$ and $p_2$ be associated with $\Omega_1$ and $\Omega_2$, respectively.
We define the sup-convolution of $p_1$ as
\beq\lb{5.3}
v_1(x,t):=(1+\delta)^{-1}\sup_{|(y,s)-(x,t)|<r -ht}p_1(y,(1+\delta)^{-2}s),
\eeq
and the inf-convolution of $p_2$ as
\beq\lb{5.4}
v_2(x,t):=(1-\delta)^{-1}\inf_{|(y,s)-(x,t)|<r-ht}p_2(y,(1-\delta)^{-2}s).
\eeq
It is direct to see
\[
U_1=\{v_1>0\}\cap \R\times (0,T)\quad\text{and}\quad U_2=\{v_2>0\}\cap \R\times (0,T).
\]

Finally, for $i=1,2$, let $u_i$ be the function associated with $U_i$.

\begin{lemma}[Associated functions of the regularizations]\lb{L.4.1}
There exists $C=C(T)$, such that if $r^{\sigma}, h \leq\delta/C$, then we have $v_1\leq u_1$ and $v_2\geq u_2$.
\end{lemma}
\begin{proof}
We will only show that $v_1\leq u_1$. Indeed, since $v_1(\cdot,t)=u(\cdot,t)=0$ on the boundary of $U_1(t)$, it suffices to verify that $v_1(\cdot,t)$ is a subsolution to the elliptic equation in $U_1(t)$.  
Suppose $v_1(x,t)>0$ and $(x',t')$ is such that $v_1(x,t)=(1+\delta)^{-1}p_1(x', t')$. Then 
\[
|(x,t)-(x',(1+\delta)^{2}t')|\leq r-ht.
\]
It follows from \cite[Lemmas 5.3, 5.4]{kim2021porous} and \cite[Lemma 5.2]{kim2024regularity} that
\beq\lb{5.15}
-(v_1)_{xx}(x,t)\leq -(1+\delta)^{-1}(p_1)_{xx}(x',t')\quad\text{and}\quad (v_1)_x(x,t)=(1+\delta)^{-1}(p_1)_x(x',t'),
\eeq
where the (in)equalities hold in the viscosity sense.  Recall
\[
-(p_1)_{xx}=\frac{A_x(x)}{A(x)}(p_1)_x+\frac{F(x)}{A(x)}.
\]
Thus, we get
\begin{align*}
- (v_1)_{xx}(x,t)&\leq \frac{A_x(x')}{A(x')}(v_1)_x(x,t)+(1+\delta)^{-1}\frac{F(x')}{A(x')}\\
&=\frac{A_x(x')}{A(x')}(v_1)_x(x,t)+\frac{F(x')}{A(x')}-\frac{\delta}{1+\delta}\frac{F(x')}{A(x')} 
\end{align*}
in the viscosity sense. By the assumption, $A$ is strictly positive, and $A_x$ is uniformly bounded and Lipschitz continuous. Also, since $p_1$ and then $v_1$ are uniformly Lipschitz continuous, and $|x-x'|\leq r$, we get
\[
\left|\frac{A_x(x')}{A(x')}(v_1)_x(x,t)+\frac{F(x')}{A(x')}-\frac{A_x(x)}{A(x)}(v_1)_x(x,t)-\frac{F(x)}{A(x)}\right|\leq Cr^{\sigma}
\]
when $t\leq T$. Also using that $\frac{\delta}{1+\delta}\frac{F(x')}{A(x')}$ is strictly positive, if $r^\sigma,h\leq \delta/C$ for some $C>0$ depending only on $T$ and other constants from the assumption,  
\[
-\partial_x(A(x) \partial_x v_1(x,t))< F(x)
\]
in the viscosity sense. Thus, we conclude that $v_1\leq u_1$. 
\end{proof}

 If a free boundary point is regular from the interior (or exterior), we refer to a tangent line of one interior (or exterior) ball at the point as a interior (or exterior) tangent line.

\begin{proposition}[Velocity at regular boundary points]\lb{P.4.2}
Let $T>0$ and let two admissible sets $\Omega_1$ and $\Omega_2$  be, respectively, viscosity sub- and superflows to \eqref{main} 
in $\R\times (0,T)$. Let $U_1$ and $U_2$ be given by \eqref{5.1} and \eqref{5.2}. There exists $C_0=C_0(T)\geq 1$ sufficiently large such that the following holds for all $r,\delta, h\in (0,1)$ satisfying
\beq\lb{cond3}
h\geq C_0\,r^2/\delta,\quad \delta/C\geq r^\sigma\quad\text{and}\quad r>2Th,
\eeq
where $C$ is as in \eqref{cond5}.
\begin{enumerate}[(i)]
    \item \label{item:reg int} If $(x_0,t_0)\in\partial U_1$, then $(x_0,t_0)$ is regular from the interior.
    \item \label{item:int tan}  Furthermore, suppose that $x_0$ is a right-hand side free boundary point such that all interior tangent lines at $(x_0,t_0)$ are given by $ (t-t_0)-k(x-x_0)=0$ with $k\geq 0$. Then there exists an interior tangent line $ (t-t_0)-k_1(x-x_0)=0$ with $k_1>0$ such that
\[
V(x_0,-\alpha_1)\geq 1/k_1+h/4\quad\text{ where $\alpha_1:=-(u_1)_x(x_0,t_0)$.}
\]

     \item \label{item:reg ext} If $(x_0,t_0)\in\partial U_2$, then $(x_0,t_0)$ is regular from the exterior.
    \item \label{item:ext tan}  Furthermore, suppose that $x_0$ is a right-hand side free boundary point such that all exterior tangent lines at $(x_0,t_0)$ are given by $ \beta(t-t_0)-(x-x_0)=0$ with $\beta\geq 0$. Then there exists an exterior tangent line $ \beta_1(t-t_0)-(x-x_0)=0$ with $\beta_1> 0$ such that
\[
V(x_0,-\alpha_2)\leq \beta_1-h/4\quad\text{ where $\alpha_2:=-(u_2)_x(x_0,t_0)$.}
\] 
Furthermore, suppose that $x_0$ is a right-hand side free boundary point such that all exterior tangent lines at $(x_0,t_0)$ are given by $ \beta(t-t_0)-(x-x_0)=0$ with $\beta\geq 0$. Then there exists an exterior tangent line $ \beta_1(t-t_0)-(x-x_0)=0$ with $\beta_1\geq 0$ such that
\[
V(x_0,-\alpha_2)\leq \beta_1-h/4\quad\text{ where $\alpha_2:=-(u_2)_x(x_0,t_0)$.}
\] 

\end{enumerate}

\end{proposition}

\begin{proof}
Let us only prove the result for $U_1$, under the assumption  that $x_0$ is a right-hand side free boundary point.  Then $\alpha_1=-(u_1)_x(x_0,t_0)>0$.
Let $(y_0,(1+\delta)^2s_0)\in\partial\Omega_1^\delta$ be such that
\beq\lb{5.6}
\left|(y_0,(1+\delta)^2s_0)-(x_0,t_0)\right|=r-ht_0.
\eeq
Since $x_0$ is one right-hand side boundary of $U_1(t_0)$ and $U_1(\cdot)$ is non-decreasing, the definition of the set $U_1$ yields  
\begin{equation*}
    y_0\leq x_0\quad\text{ and }\quad (1+\delta)^2s_0\geq t_0.
\end{equation*}
Actually, we will show in the proof that $y_0<x_0$.

Note that the space-time ball centered at $(x_0,t_0)$ with radius $r-ht_0$ is one exterior attaching ball of $(y_0,(1+\delta)^2s_0)$ in $\Omega_1^\delta$. Thus, $(y_0,(1+\delta)^2s_0)$ is regular  from exterior of $\Omega_1^\delta$ and $(y_0,s_0)$ is regular  from exterior of $\Omega_1$.

Now, let
\[
A_1:=\left\{(x,t)\mid  |(x,t)-(y_0,(1+\delta)^2s_0)|\leq r-ht\right\}.
\]
We have $A_1\subset U_1$ and $(x_0,t_0)\in(\partial A_1)\cap (\partial U_1)$. Hence, $(x_0, t_0)$ is regular from the interior. This completes the proof of part \ref{item:reg int}.

A direct computation yields that the tangent line of $\partial A_1$ at $(x_0,t_0)$ is given by
\[
\left({(1+\delta)^2s_0-t_0-(rh-h^2t_0)},{x_0-y_0}\right)\cdot(x-x_0,t-t_0)=0.
\]
By  assumption, we can rewrite this tangent line as $(x-x_0)-k(t-t_0)=0$ with $k\geq 0$. The previous expression for the tangent line therefore yields
\beq\lb{667}
\frac{x_0-y_0}{(1+\delta)^2s_0-t_0-(rh-h^2t_0)}=k\geq 0,
\eeq
which implies, in particular,
\beq\lb{666}
{(1+\delta)^2s_0-t_0-(rh-h^2t_0)}>0.
\eeq

\begin{figure}
\label{fig:section3}
    \centering
    \includegraphics[scale=1]{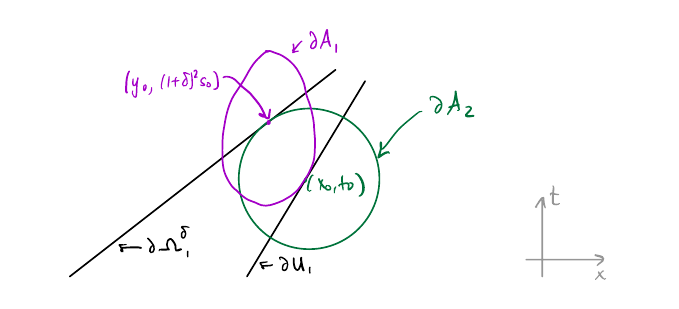}
\caption{Illustration of the space-time regions in the proof of Proposition \ref{P.4.2}.}
\end{figure}

On the other hand, since $(y_0,(1+\delta)^2s_0)\in\partial\Omega_1^\delta$, the space-time set
\[
A_2:=\{(y,s)\mid  |(x_0,t_0)-(y,(1+\delta)^2s)|\leq r-ht_0\}
\]
is outside $\Omega_1$ and $(y_0,s_0)\in\partial\Omega_1$. See Figure \ref{fig:section3} for a cartoon illustration. By \eqref{666}, $(1+\delta)^2s_0-t_0>0$. Then the tangent line of $\partial A_2$ at $(y_0,s_0)$ is given by
\[
\left(1,\frac{x_0-y_0}{(1+\delta)^2((1+\delta)^2s_0-t_0)}\right)\cdot(y-y_0,s-s_0)=0.
\]
We can rewrite the tangent line as $(y-y_0)-k'(s-s_0)=0$ where 
\[
k':=\frac{x_0-y_0}{(1+\delta)^2((1+\delta)^2s_0-t_0)}\geq 0.
\]
Letting $\alpha':=-(p_1)_x(y_0,s_0)\geq 0$,
 the first part of Proposition \ref{P.3.9} yields that $k'>0$ and $V(y_0,-\alpha')\geq (k')^{-1}$. 
Therefore, we find $x_0>y_0$, which implies that $k$ from \eqref{667} is positive, and 
\[
V(y_0,-\alpha')\geq \frac{(1+\delta)^2((1+\delta)^2s_0-t_0)}{x_0-y_0}.
\]

We now seek to bound the right hand side of the previous line from below. To this end, we note that the definition of $k$ implies
$
\frac{(1+\delta)^2s_0-t_0 }{x_0-y_0}-\frac{rh-h^2t_0}{x_0-y_0}=\frac{1}{k}$.
Multiplying by $(1+\delta)^2$ and rearranging yields,
\[
\frac{(1+\delta)^2\left((1+\delta)^2s_0-t_0 \right)}{x_0-y_0}=\frac{(1+\delta)^2}{k}+
\frac{(1+\delta)^2 h(r-ht_0)}{x_0-y_0}\geq \frac{(1+\delta)^2}k+\frac{h}{2},
\]
where we've applied $0<x_0-y_0\leq r$ as well as $r\geq 2ht_0$ to obtain the inequality. 
Thus we find,
\beq\lb{5.16}
V(y_0,-\alpha')\geq \frac{(1+\delta)^2}k+\frac{h}{2}.
\eeq

Now, by Lemma \ref{L.4.1}, $u_1\geq v_1$. And, since $(x_0,t_0)\in \partial U_1$, we have $u_1(x_0,t_0)=v_1(x_0, t_0)$, so we find
\[
\alpha_1=-(u_1)_x(x_0,t_0)\geq -(v_1)_x(x_0,t_0).
\]
Then we note that,  by \eqref{5.15}, we have $(v_1)_x(x_0,t_0)=(1+\delta)^{-1}(p_1)_x(y_0,s_0)$. So we obtain,
\[
\alpha_1=-(u_1)_x(x_0,t_0)\geq -(v_1)_x(x_0,t_0)=(1+\delta)^{-1}\alpha'.
\]
Using condition \eqref{cond5} of Assumption \ref{assump2} \ref{item:assump1.1 2} with $\gamma =\delta$, $x_1=x_0$, $x_2=y_0$, 
and $q_1=\alpha_1$, followed by the fact that  $|x_0-y_0|\leq r$, which holds by \eqref{5.6}, we get
\begin{align*}
(1+\delta)^{-2}V(y_0, -(1+\delta)\alpha_1) \leq V(x_0,
-\alpha_1 )+{C|x_0-{y_0}|^2}/{\delta}\leq V(x_0,-\alpha_1)+{Cr^2}/{\delta}.  
\end{align*}
Then, if $h\geq 4C(1+\delta)^{2}r^2/\delta$ for $C$ sufficiently large, we get from \eqref{5.16} that
\[
V(x_0,-\alpha_1)\geq 1/k+h/4,
\]
which finishes the proof.

\end{proof}

\subsection{Comparison and well-posedness}
\label{ss:comp and wp}
Now we use the regularizations that we defined in the previous subsection to prove the comparison principle and the well-posedness of viscosity solutions. 

\begin{theorem}[Comparison principle]\lb{L.cp}
Under Assumption \ref{assump2}, let two admissible  sets $\Omega_1$ and $\Omega_2$  be, respectively,  viscosity sub- and superflows to \eqref{main} 
in $\R\times (0,T)$ such that $\overline{\Omega_1(0)}\subset\Omega_2(0)$. Then $\overline{\Omega_1(t)}\subset\Omega_2(t)$ for all $t\in (0,T)$. 
\end{theorem}

\begin{proof}
For $\eps\in (0,1)$  sufficiently small, set
\[
r:=\eps^{1/\sigma+2},\quad \delta:=\eps,\quad h:=\eps^{1/\sigma+3},
\]
and then the condition \eqref{cond3} holds.
Next, let $U_1$ and $U_2$ be as in \eqref{5.1} and \eqref{5.2}. By the definition of viscosity sub- and superflow, $\Omega_1$ and $\Omega_2$ are space-time upper- and lower- semicontinuous. Thus, since $\overline\Omega_1(0)\subset\Omega_2(0)$, we have 
\[
\overline U_1(t)\subset U_2(t)
\]
for $\delta=\eps$ sufficiently small and $t$  sufficiently small.

Now, we assume for contradiction that this inclusion does not hold for all $t>0$. Thus, there exists $t_0>0$ such that $\overline U_1(t)\subset U_2(t)$ for all $t<t_0$, and there is $x_0\in\partial U_1(t_0)\cap \partial U_2(t_0)$. Assume that $x_0$ is a right-hand side free boundary point.

Let $u_i$ be the associated function of $U_i$. Since ${U_1}(t)\subseteq U_2(t)$ for $t\leq t_0$, the comparison principle implies $u_1(\cdot,t)\leq u_2(\cdot,t)$ for $t\leq t_0$. Consequently, 
\beq\lb{5.8}
-(u_1)_x(x_0,t_0)\leq -(u_2)_x(x_0,t_0),\quad (u_1)_x(x_0,t_0)<0, \quad\text{and } (u_2)_x(x_0,t_0)<0.
\eeq

By Proposition \ref{P.4.2} parts \ref{item:reg int} and \ref{item:reg ext},  
the free boundary point $(x_0, t_0)$ is regular from the interior of $U_1$ and it is regular from the exterior of $U_2$. 
Next we seek to verify the hypotheses of Proposition \ref{P.4.2} parts \ref{item:int tan} and \ref{item:ext tan}. To this end, 
write one arbitrary  interior tangent line of $\partial U_1$ at $(x_0,t_0)$ as $\beta(t-t_0)-k(x-x_0)=0$, and one exterior tangent line of $\partial U_2$  at $(x_0,t_0)$ as $\beta'(t-t_0)-k'(x-x_0)=0$. Since that $\Omega_i(\cdot)$ is non-decreasing, It follows from the definition of $U_i$ (see \eqref{5.1}, \eqref{5.2} and the discussions after) that 
\[
\beta\geq -kh\text{ with } k> 0\quad\text{and}\quad k'\geq -\beta'h\text{ with } \beta'> 0.
\]
Hence, recalling that $U_1(t_0)\subseteq U_2(t_0)$ and $(x_0,t_0)$ is the first intersection of the two sets $\partial U_1(\cdot)$ and $\partial U_2(\cdot)$, we must have that both tangent lines have positive slopes and the slope of the exterior tangent line of $U_2$ is no smaller than the one of the interior tangent line of $U_1$. Therefore, we are able to write any one interior tangent line of $U_1$ at $(x_0,t_0)$ as 
\[
(t-t_0)-k(x-x_0)=0
\]
and any one exterior tangent line of $U_2$ at $(x_0,t_0)$ as
\[
\beta(t-t_0)-(x-x_0)=0
\]
with $\beta,k\geq 0$ satisfying $\beta k\leq 1$ (due to the order of the slopes).

Now, it follows from Proposition \ref{P.4.2} that we can find a particular interior tangent line of $U_1$ with parameter $k>0$, and one particular exterior tangent line of $U_2$ with parameter $\beta\geq 0$ such that $\beta k\leq 1$, and
\[
V(x_0, (u_1)_x(x_0,t_0))\geq 1/k+h/4,
\]
and
\[
V(x_0, (u_2)_x(x_0,t_0))\leq \beta-h/4.
\]
Since $\beta k\leq 1$, we obtain
\[
V(x_0, (u_1)_x(x_0,t_0))>
V(x_0, (u_2)_x(x_0,t_0)).
\]
However, this contradicts with the monotonicity assumption on $V$ and \eqref{5.8}, which finishes the proof.
\end{proof}

Now, we will use the comparison principle and Perron's method to prove:

\begin{theorem}[Well-posedness for flows]\lb{T.3.4}
Under Assumption \ref{assump2}, let $\Omega_0\subseteq\R$ be open and bounded. Then there exists a unique open set  $\Omega$ that is a  viscosity flow of the equation \eqref{main} such that $\Omega(0)=\Omega_0$.

\end{theorem}
\begin{proof}
We define $\Omega$ to be the union of all $U$ where  $U$ is  any viscosity subflow of \eqref{main} with initial $U(0)\subseteq\Omega_0$.
The set of viscosity subflows is certainly non-empty as $U:=\Omega_0\times (0,T)$ belongs to it.
According to the definition of viscosity flow, it is easy to check that $\Omega$ is itself a subflow.

We aim to show that $\Omega_\#$ is a superflow. Suppose that this is not the case. For $p$ being the associated function of $\Omega_\#$, suppose that there exists a smooth function $\phi(x,t)$ such that
$p-\phi$ has a local minimum zero in $\overline{\Omega_\#}\cap\{t\leq t_0\}$ at $(x_0,t_0)\in\partial \Omega_\#$, and
\begin{align}
   &
    |\phi_x(x_0,t_0)|\neq 0, \quad\text{and}\lb{5.11}\\
&\phi_t(x_0,t_0)-V(x_0,\phi_x(x_0,t_0))|\phi_x(x_0,t_0)|< 0.\lb{5.111}
\end{align}

Without loss of generality, after shifting, we assume $x_0=t_0=0$, and we also assume that $x_0$ is a right-hand side boundary point. Thus, we must have $\phi_x(0,0)<0$.  
For some $\delta,\gamma\in (0,1)$, set
\[
P(x,t):=\phi(x,t)+\delta-\gamma (|x|^2+t^2).
\]
The continuity of the coefficients, together with \eqref{5.11} and \eqref{5.111}, implies that, for $r$ and $\gamma$ sufficiently small,
\beq\lb{5.122}
P_x(x,t)<0\quad\text{and}\quad P_t(x,t)-V(x,P_x(x,t))|P_x(x,t)|< 0.
\eeq
holds for all $(x,t)\in Q_r:=\{|x|+|t|<r\}$. 
We then select $\delta:=\gamma r^2/8$, and consequently, $\delta-\gamma(|x|^2+t^2)\leq 0$ for $(x,t)\in Q_{r/2}^c$. 
Since $p\geq\phi$ 
locally near $(0,0)$ for $t\leq 0$, we get $P\leq 0$ outside $\{p>0\}\cup Q_{r/2}$.

Now, consider 
\beq
\label{eq:Omega tilde}
\tilde\Omega:= \Omega\cup \{P>0\}.
\eeq
Since $P(0,0)>0$, $\tilde\Omega$ contains $\Omega$ and is strictly larger. 
To complete the proof, it suffices to show that $\tilde\Omega$ is a subflow, as this will yield the desired contradiction with the definition of $\Omega$, which will in turn imply that $\Omega$ is a viscosity superflow, as desired.

To prove that $\tilde\Omega$ is a subflow, let $\tilde p$ be the associated function to $\tilde\Omega$. In view of \eqref{eq:Omega tilde}, we have
\[
\tilde p\geq\max\{ p, P\}.
\]
Then take any test function $\tilde\phi$ that touches $\tilde p$ from above at a free boundary  point $(x_1,t_1)\in\partial\tilde\Omega$ in ${\tilde\Omega}\cap\{t\leq t_1\}$. It suffices to consider the case of $(x_1,t_1)\in\partial\{P>0\}$, as $\Omega$ is known to be a subflow. It is clear that $\tilde\phi\geq\tilde p\geq P$ locally near $(x_1,t_1)$ for $t\leq t_1$. Since $\tilde\phi(x_1,t_1)=P(x_1,t_1)=0$, this implies
\[
\tilde\phi_t(x_1,t_1)\leq P_t(x_1,t_1)\quad\text{and}\quad -\tilde\phi_x(x_1,t_1)\geq -P_x(x_1,t_1)>0,
\]
where the last strict inequality is due to \eqref{5.122}. 
Recall that $P\leq 0$ outside $\{p>0\}\cup Q_{r/2}$, and so $(x_1,t_1)\in Q_{r/2}$. Thus, \eqref{5.122} and the monotonicity of $V$ yield
\[
\tilde \phi_t(x_1,t_1)-V(x_1,\tilde \phi_x(x_1,t_1))|\tilde \phi_x(x_1,t_1)|< 0.
\]
We conclude that $\tilde\Omega$ is a subflow.

Next, we prove uniqueness. Suppose that $\Omega_1$ and $\Omega_2$ are two viscosity flows. We claim that
\[
\overline{\Omega_2(t)}\subseteq\Omega_2(t+s)
\]
for any $s> 0$.
Indeed, suppose that $[a,b]\subseteq \Omega_2(t)$. It suffices to show that $(a-\eps s,b+\eps s)\subseteq \Omega_2(t+s)$ for all $s\in (0,1)$ and for some $\eps>0$ that depends only on $b-a$ and  universal constants. By the comparison principle (Theorem \ref{L.cp}), it suffices to show that $(a-\eps s,b+\eps s)$ is a subflow. Let $u(x,s)$ be the associated function of the set $(a-\eps s,b+\eps s)$, i.e., $-\partial_x(A u_x)=F$ in $(a-\eps s,b+\eps s)$ with $0$-boundary condition. Then there exists $C\geq 1$, depending only on universal constants and $b-a$, such that for all $s\in (0,1)$ and $\eps<b-a$,
\[
-u_x(b+\eps s,s),\, u_x(a-\eps s,s)\in [1/C,C]\quad\text{and}\quad
u_t(x,s)\in [0,C\eps].
\]
Here $u_x(\cdot,s)$  is defined as the left (or right) limit of $u_x$  evaluated at $x=b+\eps s$ (or $x=a-\eps s$).
Thus, by taking $\eps$ to be sufficiently small and applying Assumption \ref{assump2} \ref{item:assump1.1 2}, we have
\[
u_t(x,s)\leq V(\cdot,u_x)|u_x|(x,s)\quad\text{ for }x=a-\eps s\text{ or }b+\eps s
\]
in the classical sense. Thus, we conclude that $(a-\eps s,b+\eps s)$ is a subflow for all $s\in (0,1)$, which finishes the proof of the claim.

Now, since the coefficients of \eqref{main} are time-independent, $\Omega'(\cdot):=\Omega_2(\cdot+s)$ is also a viscosity flow. Then using $\overline{\Omega_1(0)}\subseteq\Omega'(0)$, Theorem \ref{L.cp} implies that
\[
\overline{\Omega_1(t)}\subseteq\Omega_2(t+s)
\]
for all $s$. By the continuity of viscosity flows and passing $s\to 0$, we obtain $\Omega_1=\Omega_2$.
\end{proof}

Finally, we provide:
\begin{proof}[Proof of Theorem \ref{T.1}] By Theorem \ref{T.3.4} and 
 Lemma \ref{L.def}, there exists a unique viscosity solution to \eqref{main} with initial data having support of $\Omega_0$.
\end{proof}

\subsection{Well-posedness of \eqref{1.1}}
\label{ss:wp ep}
 To end the section, we show that the equation \eqref{1.1} is also well-posed. Indeed, this is a direct corollary of Theorem \ref{T.3.4} and the following lemma.
\begin{lemma}[Well-posedness of the equation with oscillation]
\label{lem:assumption ep}
Under Assumption \ref{assumption} (1)--(3), for each fixed $\omega$, if taking
\[
V(x,q):=B(x,\eps^{-1}x,\omega)|q|+G(x,\eps^{-1}x,\omega),
\]
Assumption \ref{assump2} \ref{item:assump1.1 2} is satisfied with $C$ depending on $\eps$.    
\end{lemma}
\begin{proof}
It is direct to check by the Lipschitz condition on $B$ and $G$ that \eqref{cond4} holds, and
let us show \eqref{cond5} below. Note that, since $\sqrt{G(x,y,\omega)}$ is uniformly Lipschitz continuous in both $x$ and $y$, there exists $C=C(\eps)$ such that
\[
|G(x_1,\eps^{-1}x_1,\omega)-G(x_2,\eps^{-1}x_2,\omega)|
\leq C\sqrt{G(x_1,\eps^{-1}x_1,\omega)}|x_1-x_2|+C|x_1-x_2|^2.
\]
Hence, since $B$ is Lipschitz continuous and strictly positive, and $|x_1-x_2|\leq C_0\gamma$ for some $C_0>0$, we obtain for some $C>0$ independent of $C_0$ that
\beq\lb{5.7}
\begin{aligned}
V(x_1,q_1)&
\geq V(x_2,q_1)-C|x_1-x_2||q_1|-|G(x_1,\eps^{-1}x_1,\omega)-G(x_2,\eps^{-1}x_2,\omega)|\\
&\geq V(x_2,q_1)-C\gamma|q_1|/C_0-C\sqrt{V(x_1,q_1)}|x_1-x_2|-C|x_1-x_2|^2.  
\end{aligned}
\eeq
Next, if $C_0$ is sufficiently large,
\[
(1+\gamma)V(x,q)\geq V(x,(1+\gamma)q)\quad\text{and}\quad\frac{\gamma}{2}V(x_1,q_1)\geq 2C\gamma|q_1|/C_0.
\]
Consequently, it follows from \eqref{5.7} that
\begin{align*}
&(1+\gamma)^2 V(x_1,q_1)+C'|x_1-x_2|^2/\gamma-V(x_2,(1+\gamma)q_1)\\
&\qquad \geq (1+3\gamma/2) V(x_1,q_1)+2C\gamma|q_1|/C_0+C'|x_1-x_2|^2/\gamma-(1+\gamma)V(x_2,q_1)\\
&\qquad \geq \gamma V(x_1,q_1)/2+C'|x_1-x_2|^2/\gamma-2C\sqrt{V(x_1,q_1)}|x_1-x_2|-2C|x_1-x_2|^2
\end{align*}
which is nonnegative when $C'$ is sufficiently large. We proved the claim.
\end{proof}
We will show in Lemma \ref{L.4.8} that $V=\ol V$ with $\ol V$ from  Theorem \ref{thm:main} satisfies Assumption \ref{assump2} \ref{item:assump1.1 2} as well, under the extra assumption of strictly positive $G$.

\section{The random environment and interior homogenization} \label{S.random}

\subsection{The random environment}
In this section, we gather several theorems from probability.

\begin{theorem}[Birkhoff--Khinchin Theorem]
Suppose that the transformation group $\{\tau_y\}$ satisfies the ergodicity assumption, and $f:\bbR\times\Sigma\to \bbR$ is stationary and uniformly finite. Then there is $\Sigma_0\subseteq\Sigma$ with $\bbP[\Sigma_0]=1$ such that for all $\omega\in\Sigma_0$ and all $x_0\in\bbR$ we have
\[
\lim_{r\to\infty}\frac{1}{r}\int_0^r f(x_0+x,\omega)dx=\lim_{r\to\infty}\frac{1}{r}\int_0^r f(x_0,\tau_x\omega)dx=\bbE[f(0,\cdot)].
\]
\end{theorem}
\begin{proof}
If $x_0=0$, the conclusion follows from the continuous version of Birkhoff--Khinchin Theorem \cite{kallenberg1997foundations,bergelson2012discrete}. The general case follows from the observation that
\[
\frac{1}{r}\int_0^r f(x_0+x,\omega)dx-\frac{1}{r}\int_0^r f(x,\omega)dx=\frac{1}{r}\left(\int_r^{r+x_0}f(x,\omega)dx-\int_0^{x_0}f(x,\omega)dx\right),
\]
which goes to $0$ as $r\to\infty$.
\end{proof}

\begin{theorem}[Wiener's Ergodic Theorem \cite{wiener1939ergodic,becker1981multiparameter}]
Suppose that $\{\tau_y\}$ satisfies the ergodicity assumption. For each $g\in L^1(\Sigma)$, there is $\Sigma_0\subseteq\Sigma$ with $\bbP[\Sigma_0]=1$ such that for all $\omega\in\Sigma_0$,
\[
\lim_{r\to\infty}\frac{1}{|\calB_r|}\int_{\calB_r}g(\tau_y\omega)dy=\int_\Sigma g(\omega)d\bbP.
\]
\end{theorem}

\begin{theorem}[Kingman's Subadditive Ergodic Theorem \cite{kingman1968ergodic}]
Suppose that $\{\tau_y\}$ satisfies the ergodicity assumption.  Let $g_n\in L^1(\Sigma)$ be a sequence of measurable functions on $(\Sigma,\calF,\bbP)$ such that
\[
\inf_{n\in\bbN}\frac1n \int_\Sigma g_n(\omega)d\bbP>-\infty,
\]
and for any $n,m\in\bbN$,
\[
g_{m+n}(\omega)\leq g_n(\omega)+g_m(\tau_n\omega).
\]
Then $\frac{1}{n}g_n$ converges almost everywhere on $\Sigma$ to  a constant. 
\end{theorem}

The term $\int_{z-M}^z\frac{1}{a(y,\omega)dy}$ in the above formula satisfies an asymptotic property. Indeed, since $a(x,\omega)$ is positive and strictly away from both $0$ and $\infty$, and stationary ergodic,
it follows from the Birkhoff--Khinchin Theorem that
\beq\lb{2.10}
\lim_{M\to \infty}M^{-1}\int_{z-M}^z\frac{1}{a(y,\omega)}dy=\lim_{M\to \infty}M^{-1}\int_{0}^M\frac{1}{a(z-y,\omega)dy=}\frac1{\ol a}
\eeq
where $\ol a$ is a constant that is independent of both $z$ and $\omega$.

Moreover, we have the following uniform convergence. This will be crucial in the proof of existence of deterministic homogenized boundary velocity. The proof is essentially similar to the one of Proposition 3.4 \cite{lin2019stochastic}, though the problems appear to be very different.

\begin{lemma}[Uniform convergence to the mean]\lb{L.2.2} 
Suppose $a: \R\times \Sigma \rightarrow \R$ is uniformly positive and finite, and satisfies the stationary ergodic assumption. 
Then, for any $R>0$, there exists $\Sigma_0\subseteq\Sigma$ of full measure such that for every $\omega\in\Sigma_0$,  
\beq\lb{2.3}
\lim_{M\to \infty}\sup_{ |z|\leq MR}\left|M^{-1}\int_{-M}^0\frac{1}{a( y,\tau_z\omega)}dy-\frac1{\ol a}\right|=0.
\eeq
\end{lemma}
\begin{proof}
Let $\Sigma'\subseteq\Sigma$ be such that $\bbP[\Sigma']=1$ and for each $\omega\in\Sigma'$ we have that \eqref{2.10} holds.
Then by Egorov's Theorem, for any $\delta>0$, there exists $M_\delta>0$ and $D_\delta\subseteq\Sigma$ such that for all $\omega\in D_\delta $ and $M\geq M_\delta$,
\[
\bbP[D_\delta ]\geq 1-\delta \quad\text{and}\quad \left| M^{-1}\int_{-M}^0\frac{1}{a(y,\omega)}dy-\frac1{\ol a}\right|\leq \delta.
\]
Wiener's Ergodic Theorem yields that there exists $\Sigma_\delta\subseteq\Sigma$ with $\bbP[\Sigma_\delta]=1$ such that for all $\omega\in\Sigma_\delta$ we have
\[
\lim_{r\to\infty}\frac{1}{|\calB_r|}\int_{\calB_r}\chi_{D_\delta }(\tau_y\omega)dy=\bbP[D_\delta ]\geq 1-\delta.
\]
Therefore, for each $\omega\in\Sigma_\delta$ and $R>0$, we can find $M_{\delta,\omega}$ such that for all $M\geq M_{\delta,\omega}$,
\beq\lb{2.4}
\left|\{z\in \calB_{2MR}\mid \tau_z\omega\in D_\delta \}\right|\geq (1-2\delta)4MR.
\eeq

Now we fix $R>0$, $\omega\in \Sigma_\delta$ and $M\geq M_{\delta,\omega}$. Take any $x\in \calB_{MR}$, and then it follows from \eqref{2.4} that there exists $z\in \calB_{2MR}$ such that $|z-x|\leq 4\delta MR$ and $\tau_z\omega\in D_\delta $. On the other hand, note that
\begin{align*}
&\left|\int_{-M}^0\frac{1}{a(y+x,\omega)}dy-\int_{-M}^0\frac{1}{a(y+z,\omega)}dy\right| \\
&\qquad\quad= \left|\int_{x-M}^x\frac{1}{a(y,\omega)}dy-\int_{z-M}^z\frac{1}{a(y,\omega)}dy\right|\leq C|x-z|,    
\end{align*}
where, to obtain the inequality, we used that $a$ is uniformly strictly positive. Therefore, since $a(y,\tau_x\omega)=a(y+x,\omega)$ and $\tau_z\omega\in D_\delta $, we get
\begin{align*}
&\left| M^{-1}\int_{-M}^0\frac{1}{a(y,\tau_x\omega)}dy-\frac1{\ol a}\right|\\
&\qquad\quad\leq M^{-1}\left| \int_{-M}^0\frac{1}{a(y,\tau_x\omega)}-\frac{1}{a(y,\tau_z\omega)} dy\right|+\left| M^{-1}\int_{-M}^0\frac{1}{a(y,\tau_z\omega)}dy-\frac1{\ol a}\right|\\
&\qquad\quad\leq CM^{-1}|x-z|+\delta \leq (4CR+1)\delta.   
\end{align*}
Since this holds for all $x$ such that $|x|\leq MR$, $\omega\in\Sigma_\delta$ and $M\geq M_{\delta,\omega}$, and since $\delta>0$ is arbitrary, this finishes the proof.
\end{proof}

\subsection{Interior homogenization}
In the last part of this section, we prove interior stochastic homogenization for elliptic equations. We are concerned with the following elliptic equation
\beq\lb{8.1}
-\partial_x(A(x,\eps^{-1}x,\omega) \partial_x u_\eps)=F(x,\eps^{-1}x,\omega)\quad\text{ in }[l,r]
\eeq
for some $r>l$. The following two lemmas are classical and there are a lot of works for the case when $F(x,y,\omega)\equiv F(x,\omega)$, see, for instance \cite{armstrong2014regularity,bella2017stochastic,gu2017high}. We provide a simple proof, via the explicit formula of solutions, for the general equation but only with space dimension $1$.

\begin{lemma}[Interior homogenization]\lb{L.9.1}
Suppose that $1/A>0$ and $F\geq 0$ are uniformly finite. There exist deterministic functions $\ol A(x)$ and $\ol F(x)$,  and $\Sigma_0\subseteq\Sigma$ of full measure such that the following holds for all $\omega\in\Sigma_0$. Let $a,b,a_\eps,b_\eps\in\bbR$ and suppose
\[
\lim_{\eps\to0}a_\eps=a\quad\text{and}\quad \lim_{\eps\to0}b_\eps=b.
\]
If $u_\eps(x,\omega)$ satisfies \eqref{8.1} in $[l,r]$ with boundary data  $u_\eps(l)=a_\eps$ and $u_\eps(r)=b_\eps$, then
\[
u_\eps(\cdot,\omega) \to \ol u(\cdot)\quad\text{ in }[l,r]
\]
locally uniformly, where $\ol u(x)$ is the unique solution to 
\beq\lb{7.2}
 -\partial_x(\overline A(x) \partial_x \overline u)=\ol F(x) \quad \text{ for }x\in [l,r]
\eeq
which satisfies the boundary data $\ol u(l)=a$ and $\ol u(r)=b$.

Finally, if $A$ and $F$ satisfy Assumption \ref{assumption} (i), then $\overline{A}$ and $\ol{F}$ satisfy Assumption \ref{assump2} \ref{item:assump1.1.1}.
\end{lemma}
\begin{proof}
After integration, we get for $x\in [l,r]$,
\[\lb{9.1}
u_\eps(x,\omega)=   a_\eps+\int_l^{x}\frac{C_\eps}{A(y,\eps^{-1}y,\omega)}-\frac{\int_{l}^{y}F(z,\eps^{-1}z,\omega)dz}{A(y,\eps^{-1}y)}dy
\]
where  $C_\eps$ is determined by
\[
C_\eps:=  \left(\int_{l}^{r}\frac{dy}{A(y,\eps^{-1}y,\omega)}\right)^{-1} \left(\int_{l}^{r}\frac{\int_{l}^{y}F(z,\eps^{-1}z,\omega)dz}{A(y,\eps^{-1}y,\omega)}dy +b_\eps-a_\eps\right).
\]   
Since both $1/A(x,\cdot)$ and $F(x,\cdot)$ satisfy the stationary ergodic assumptions (for any fixed $x$), and they are uniformly bounded, we get from Birkhoff--Khinchin Theorem that for any $x>l$, and for almost all $\omega$,
\[
\lim_{\eps\to0}\int_l^{x}\frac{dy}{A(y,\eps^{-1}y,\omega)}=\int_l^{x}\frac{dy}{\ol A(y)},
\]
\[
\lim_{\eps\to0}\int_{l}^{x}F(z,\eps^{-1}z,\omega)dz=\int_{l}^{x}\ol F(z)dz
\]
where 
\begin{equation}
    \label{eq:barA barF}
{1}/{\ol A(y)}=\bbE[1/A(y,0,\omega)]\quad\text{and}\quad {\ol F(y)}=\bbE[F(y,0,\omega)].
\end{equation}
By the dominated convergence theorem, $u_\eps$ converges almost surely to 
\[
\ol u(x):=a+\int_l^{x}\frac{\ol C}{\ol A(y)}-\frac{\int_{l}^{y}\ol F(z)dz}{\ol A(y)}dy
\]
where
\[
\ol C:=  \left(\int_{l}^{r}\frac{dy}{\ol A(y)}\right)^{-1} \left(\int_{l}^{r}\frac{\int_{l}^{y}\ol F(z)dz}{A(y)}dy +b-a\right).
\]
It can be checked directly that $\ol u$ is the unique solution $\ol u$ to \eqref{7.2} satisfying $\ol u(l)=a$ and $\ol u(r)=b$.

Finally, the last claim holds due to \eqref{eq:barA barF}.
\end{proof}

If the boundary data is not fixed, we have the following result.
\begin{lemma}[Interior homogenization with variable boundary data]\lb{L.9.2}
For almost all fixed $\omega$,
suppose that $u_\eps$ satisfies \eqref{8.1} in $[l,r]$, and $u_\eps(l,\omega)$ and $u_\eps(r,\omega)$ are uniformly bounded for all $\eps>0$. Then
$
\limsup_{\eps\to 0} u_\eps(\cdot,\omega)
$
is a subsolution to \eqref{7.2} in $(l,r)$, and
$
\liminf_{\eps\to 0} u_\eps(\cdot,\omega)$
is a supersolution to \eqref{7.2} in $(l,r)$.
\end{lemma}
\begin{proof}
Let us only prove that $
u_* (\cdot,\omega):=\liminf_{\eps\to 0} u_\eps(\cdot,\omega)$ is a supersolution. Since $u_\eps$ is uniformly continuous, $u_* (x,\omega)=\liminf_{\eps\to 0} u_\eps(x,\omega)$ for all $x\in [l,r]$ is equivalent to $u_* (x,\omega)=\liminf_{\eps\to 0} u_\eps(x,\omega)$ for all rational numbers $x\in [l,r]$. Thus, there exists a sequence of $\eps_k\to 0$ such that
\[
u_* (\cdot,\omega)=\liminf_{k\to\infty} u_{\eps_k}(\cdot,\omega)\quad\text{ in }[r,l].
\]
Consequently, there are at most countably many real values $\{a_i\}_{i\in\bbN}$ that are limits of subsequences of $\{\eps_k\}_{k\in\bbN}$. We denote those subsequences as $\{\eps_k^i\}_{k\in\bbN}$, and for each $i$, we have
\[
\lim_{k\to \infty}u_{\eps_k^i}(l,\omega)= a_i.
\]
Furthermore, by passing to a subsequence of $\eps_k^i$, we can assume that
\[
\lim_{k\to\infty}u_{\eps_k^i}(r,\omega)= \liminf_{k\to\infty}u_{\eps_k^i}(r,\omega)=:b^i.
\]
By the comparison principle for linear elliptic equations, we have
\[
\lim_{k\to\infty}u_{\eps_k^i}(\cdot,\omega)= \liminf_{k\to\infty}u_{\eps_k^i}(\cdot,\omega)\quad\text{ in }[r,l].
\]
It follows that
\[
u_* (\cdot,\omega):=\inf_{i\in\bbN}\{\lim_{k\to\infty} u_{\eps_k^i}(\cdot,\omega)\}.
\]
Note that by Lemma \ref{L.9.1}, $\lim_{k\to\infty} u_{\eps_k^i}(\cdot,\omega)$ is a solution to \eqref{7.2} for each $i$ with boundary data $(a_i,b_i)$. Therefore, $u_* $ is a supersolution.
\end{proof}

\section{The effective velocity}\lb{S.4}
The main task is to identify the homogenized boundary behavior; that is, to find $\ol V$ in \eqref{1.3}. 
For this purpose, we first construct an effective equation --- an ODE --- that describes the motion of the free boundary  (Section \ref{ss:effective equation}). Then, in Section \ref{ss:effective velocity}, we show that the arrival time to any $x\in \R$ is subadditive; applying the subadditive ergodic theorem yields the homogenized arrival time,   which we then use to define the homogenized velocity $\ol V$. Finally, in Section \ref{ss:continuity V}, we establish that $\ol V$ is continuous.

Throughout the section, we work with generic coefficients the coefficients $a$, $b$, and $g$ satisfying:
\begin{assumption}
\label{assump5.1}
    We assume that the following hold for
    \[
    a, b, g: \R\times \Sigma \rightarrow \R.
    \]
\begin{enumerate}[(i)]
\item The coefficients $a$ and $b$ are uniformly positive and finite, and $g\geq 0$ is finite;
\item The coefficients $a$, $b$, and $g$ are uniformly Lipschitz continuous in $x$;
\item For each fixed $x$, the coefficients $a(x, \cdot)$, $b(x,\cdot)$, and $g(x,\cdot)$  satisfy the stationary ergodic assumption. 
\end{enumerate}
\end{assumption}

Later, we will apply the results in this section to equation \eqref{1.1}  with coefficients ``frozen" at some $x_0\in\bbR$, as in \eqref{eq:frozen}. 
\subsection{The effective equation}
\label{ss:effective equation}

\subsubsection{Definition and basic properties}
Motivated by the hueristics presented in Section \ref{sss:intro hom},  we introduce the following ODE for $S_q^{x_0}(t)$, which we expect characterizes the effective location of the free boundary of \eqref{eq:motivating pde}:
\beq\lb{FB}
\frac{d}{dt}S_{q}^{x_0}(t, \omega)=q b(S_{q}^{x_0}(t,\omega),\omega) \mu(S_{q}^{x_0}(t,\omega),\omega)+g(S_{q}^{x_0}(t,\omega),\omega)
\eeq
with initial condition  $S_{q}^{x_0}(0,\omega)=x_0$ and with the notation
\beq\lb{3.14}
 \mu(x,\omega):=\ol a /{a(x,\omega)}.
\eeq
We remark that this equation is independent of $r$. For any $y\in\bbR$, since
\[
(b(S,\tau_y\omega), \mu(S,\tau_y\omega),g(S,\tau_y\omega))=(b(S+y,\omega) ,\mu(S+y,\omega),g(S+y,\omega) )
\]
the ODE yields that
\beq\lb{2.68}
S_q^{x_0+y}(t,\omega)=S_q^{x_0}(t,\tau_y\omega).
\eeq

Standard Cauchy-Lipschitz theory, together with our assumptions on $a$ and $b$, imply:

\begin{lemma}[Well-posedness of \eqref{FB}]\lb{L.2.3}
There exists a solution $S$ to the effective equation \eqref{FB} with $S(0)=x_0$; that is, $S$ is both a subsolution and a supersolution and $S(0)=x_0$. And, \eqref{FB} enjoys the standard comparison principle for ODEs.

\end{lemma}

The following lemma describes some basic properties of $S_q^{x_0}$.
\begin{lemma}[Basic properties of \eqref{FB}]\lb{L.2.6}
Let $S_q^{x_0}$ satisfy \eqref{FB} with with initial condition  $S_{q}^{x_0}(0,\omega)=x_0$. We have
\begin{enumerate}
    \item \[
\frac{d}{dt}S_{q}^{x_0}(t)\in[ qc_{\min},qC_{\max}+\|g\|_\infty].
\]
where
\[
c_{\min}:=\inf_{z,\omega} \left\{\ol a {b(z,\omega)}/{a(z,\omega)}\right\}>0\quad\text{and}\quad C_{\max}:=\sup_{z,\omega}\left\{ \ol a {b(z,\omega)}/{a(z,\omega)}\right\}>0.
\]

\item 
For any $x_0,x_1\in\bbR$,  we have
\[
S_{q}^{x_0}(t)\leq S_{q}^{x_1}(t)+(x_0-x_1)_+(qC_{\max}+\|g\|_\infty)/(qc_{\min}).
\]
\end{enumerate}

\end{lemma}

\begin{proof}
The first claim follows directly from the ODE \eqref{FB}.
Let us only prove the second claim.

If $x_0\leq x_1$, the uniqueness of the ODE yields
$S_{q}^{x_0}(t)\leq S_{q}^{x_1}(t)$ for all $t\geq 0$. 
If $x_0>x_1$, from the first part of the claim, $S_q^{x_1}(t_0)\geq x_0$ where 
\[
t_0:=(x_0-x_1)/ (qc_{\min}).
\]
Thus, by comparing $S_{q}^{x_0}(t)$ with $S_{q}^{x_1}(t+t_0)$, we get
\[
S_{q}^{x_0}(t)\leq S_{q}^{x_1}(t+t_0)\leq S_{q}^{x_1}(t)+(qC_{\max}+\|g\|_\infty)t_0
\]
which finishes the proof.
\end{proof}

\subsection{The Effective Velocity}
\label{ss:effective velocity}
This subsection is devoted to defining the effective velocity $\ol V$ and its continuity in $x$. The effective velocity will be identified by the asymptotics of solutions to \eqref{FB}.
Our strategy in this section is to apply the subadditive ergodic theorem to the quantity that we refer to as the ``arrival times''.

For each fixed $\omega$, let $S^{x_0}_q(t,\omega)$ solve the ODE \eqref{FB} with initial condition  $S^{x_0}_q(0,\omega)=x_0$.
Define the arrival time to any $x_1>x_0$ as
\[
T_q(x_1;x_0,\omega):=\inf\{t>x_0\mid  S_q^{x_0}(t,\omega)=x_1\}.
\]
Then $T_q(\cdot;x_0,\omega)$ can be viewed as the inverse function of $S_q^{x_0}(\cdot,\omega)$. We have the following result.

\begin{lemma}\lb{L.4.5}
For each $q>0$, there exists a constant $\ol{T}_q=\ol T_q(a,b)>0$ such that for all $x_0\in\bbR$ and almost all $\omega$ we have
\beq\lb{45}
\lim_{n\to\infty}T_q(x_0+n;x_0,\omega)/n =\ol{T}_q.
\eeq
Moreover, $1/\ol{T}_q\in [ qc_{\min},qC_{\max}+\|g\|_\infty]$ where $c_{\min}$ and $C_{\max}$ are from Lemma \ref{L.2.6}. 
\end{lemma}
\begin{proof}
Note that, due to the additivity property of ODEs, we have 
for any $n,m\geq 0$, 
\begin{align*}
T_q(x_0+n+m;x_0,\omega)&=T_q(x_0+n;x_0,\omega)+T_q(x_0+n+m;x_0+n,\omega)\\
&= T_q(x_0+n;x_0,\omega)+T_q(x_0+m;x_0,\tau_n\omega),
\end{align*}
where the second equality follows since $b$, $\mu$, and $g$ are stationary.
Thus, the convergence of $T_q(x_0+n;x_0,\omega)/n$ as $\bbN\ni n\to\infty$ to some constant $\ol T_q(x_0)$ follows from Kingman's Subadditive Ergodic Theorem. Since $T_q(\cdot;x_0,\omega)$ is non-decreasing, for any $n'\in [n,n+1]$ we have
\[
\frac{T_q(x_0+n;x_0,\omega)}n \frac{n}{n+1}\leq \frac{T_q(x_0+n';x_0,\omega)}{n'}\leq \frac{T_q(x_0+n+1;x_0,\omega)}{n+1}\frac{n+1}{n}.
\]
Thus, \eqref{45} holds with $[1,\infty)\ni n\to\infty$.

Next, since 
\[
\frac{d}{dt}S_q^{x_0}(t)\in [ qc_{\min},qC_{\max}+\|g\|_\infty]
\]
by Lemma \ref{L.2.6}, 
we get for any $x_1>x_0$, 
\[
T_q(x_1+n;x_0,\omega)-\frac{x_1-x_0}{qc_{\min}}\leq {T_q(x_1+n;x_1,\omega)}\leq T_q(x_1+n;x_0,\omega).
\]
This implies that $\ol T(x_0)$ is actually independent of $x_0$. By Lemma \ref{L.2.6} again, we have for any $n>0$,
\[
qc_{\min}\leq n/T_q(x_0+n;x_0,\omega)\leq (qC_{\max}+\|g\|_\infty),
\]
which proves the second claim.
\end{proof}

Consider variables of the micro-scale: Let $S_q^{x_0}$ be a solution to \eqref{FB}, and define
\begin{align}\lb{FBeps}
    S_{q,\eps}^{x_0}(t):=\eps (S^{x_0}_q(\eps^{-1}t)-x_0)+x_0.
\end{align}
As a direct corollary of the previous lemma, we have the following result. 

\begin{corollary}[Effective velocity]\lb{P.2.12}
Let $a,b$ and $g$ satisfy (1)--(3) from Section \ref{S.4}. Let $c_{\min}$ and $C_{\max}$ be from Lemma \ref{L.2.6}. Then for any $q>0$, 
\beq\lb{speed}
\ol V_q(a,b):=1/\ol T_q(a,b)\in  [ qc_{\min},qC_{\max}+\|g\|_\infty]
\eeq
independent of $\omega$ such that 
for all $t>0$ and almost all $\omega$,
\[
\lim_{t\to\infty}\left[S_q^{x_0}(t,\omega)-x_0\right]/{t}=\ol V_q.
\]
Equivalently, we have for  almost all $\omega$,
\[
\lim_{\eps\to 0}S_{q,\eps}^{x_0}(t,\omega)=x_0+\ol V_qt.
\]
\end{corollary}
\begin{proof}

Lemma \ref{L.4.5} directly implies that \eqref{speed} holds. 

Next, we assume $x_0=0$ for simplicity.
By Lemma \ref{L.4.5} again, for any $\delta>0$ and almost all $\omega$, there exists $N_\omega$ such that for all $n\geq N_\omega$,
\[
|n/T_q(n;0,\omega)-\ol V_q|\leq\delta.
\]
In view of \eqref{speed}, this implies that for all $t\geq N_\omega/(qc_{\min})$ (then $S_q^0(t,\omega)\geq N_\omega$),
\[
|S_q^0(t,\omega)/t-\ol V_q|\leq\delta.
\]
From this, we  obtain the second claim, and the last claim follows from the definition of $S_{q,\eps}^{x_0}$. 

\end{proof}

We upgrade the previous statement to the following uniform convergence result.

\begin{proposition}[{Uniform convergence of $S_{q ,\eps}^{ x_0}$}]\lb{P.2.13}

For any $\Lambda>0$ and $t_1>t_0>0$, there exists a set $\Sigma_0\subset \Sigma$ of full measure such that for every $\omega\in \Sigma_0$,
\[
\lim_{\eps\to 0} \sup_{ {|y|\leq\Lambda/\eps} \,\&\, t\in [t_0,t_1]}\left|S_{q ,\eps}^{ x_0}(t,\tau_y\omega)-x_0-\ol V_qt\right|=0.
\]
\end{proposition}
\begin{proof}
By the definition of $S_{q,\eps}^{x_0}$, 
\[
S_{q,\eps}^{x_0}(t,\omega)=\eps \left(S^{x_0}_{q }(\eps^{-1}t,\omega)-x_0\right)+x_0=t\left((S_{q,\eps /t}^{x_0}(1,\omega)-x_0\right)+x_0,
\]
so we have, for any $\ep>0$ and $t\in [t_0,t_1]$,
\begin{align*}
    \sup_{ |y|\leq \Lambda/\eps }\left|S_{q ,\eps}^{ x_0}(t,\tau_y\omega)-x_0-\ol V_qt\right|
    &= \sup_{ |y|\leq \Lambda/\eps }
    t\left|S_{q ,\eps/t}^{ x_0}(1,\tau_y\omega)-x_0-\ol V_q\right|\\
    &\leq  \sup_{ |y|\leq \Lambda/(\eps' t_0)}
    t_1\left|S_{q ,\eps'}^{ x_0}(1,\tau_y\omega)-x_0-\ol V_q\right|,
\end{align*}
where $\ep'=\ep/t$; note $\ep'\in [\ep/t_1, \eps/t_0]$. 
Thus, to prove the proposition, it suffices to show
\beq\lb{2.7}
\lim_{\eps\to 0} \sup_{|y|\leq\Lambda/\eps}\left|S_{q ,\eps}^{ x_0}(1,\tau_y\omega)-x_0-\ol V_q\right|=0\quad a.e.\,\,\omega.
\eeq
By Corollary \ref{P.2.12}, \eqref{2.7} holds when $\Lambda=0$. The rest of the proof is similar to the one of Lemma \ref{L.2.2}, except that we will further use Lemma \ref{L.2.6} (2).

By Egorov's Theorem, for any $\delta>0$, there exists $\eps_\delta>0$ and $D_\delta\subseteq\Sigma$ such that for all $\omega\in D_\delta $ and $\eps\in (0,\eps_\delta)$,
\[
\bbP[D_\delta ]\geq 1-\delta \quad\text{and}\quad \left|S_{q ,\eps}^{ x_0}(1,\omega)-x_0-\ol V_q\right|\leq \delta.
\]
As before, we apply Wiener's Ergodic Theorem to get $\Sigma_\delta\subseteq\Sigma$ with $\bbP[\Sigma_\delta]=1$ such that for all $\omega\in\Sigma_\delta$,
\[
\lim_{r\to\infty}\frac{1}{|\calB_r|}\int_{\calB_r}\chi_{D_\delta }(\tau_y\omega)dy=\bbP[D_\delta ]\geq 1-\delta.
\]
Therefore, for each $\omega\in\Sigma_\delta$ and $\Lambda>0$, there exists $\eps_{\delta,\omega}$ such that if $\eps\in (0,\eps_{\delta,\omega})$,
\beq\lb{2.69}
\left|\{z\in \calB_{2\Lambda/\eps}\mid \tau_z\omega\in D_\delta \}\right|\geq (1-2\delta)4\Lambda/\eps.
\eeq

Let us fix $R>0$, $\omega\in \Sigma_\delta$ and $\eps\in (0,\eps_{\delta,\omega})$, and take an arbitrary $y\in \calB_{\Lambda/\eps}$. It follows from \eqref{2.69} that there exists $z\in \calB_{2\Lambda/\eps}$ such that 
\[
|y-z|\leq 4\delta \Lambda/\eps
\]and $\tau_z\omega\in D_\delta $. 
By Lemma \ref{L.2.6} (2), we get
\[
\left|S_{q}^{x_0+y}(1,\omega)- S_{q}^{x_0+z}(1,\omega)\right|\leq |y-z|(qC_{\max}+\|g\|_\infty)/(qc_{\min})\leq \frac{4\delta\Lambda}{\ep}(qC_{\max}+\|g\|_\infty)/(qc_{\min}).
\]
By \eqref{2.68},  $S_{q}^{x_0}(1,\tau_y\omega)=S_{q}^{x_0+y}(1,\omega)$.
Thus, also using that $\tau_z\omega\in D_\delta $, we obtain
\begin{align*}
 \left|S_{q,\eps}^{x_0}(1,\tau_y\omega)-x_0-\ol V_q\right|
&\leq \left|S_{q,\eps}^{x_0}(1,\tau_y\omega)-S_{q,\eps}^{x_0}(1,\tau_z\omega)\right|+\left| S_{q,\eps}^{x_0}(1,\tau_z\omega)-x_0-\ol V_q\right|\\
&\leq \eps\left|S_{q}^{x_0+y}(\eps^{-1},\omega)-S_{q}^{x_0+z}(\eps^{-1},\omega)\right|+\delta\\
&\leq (4\Lambda(qC_{\max}+\|g\|_\infty)/(qc_{\min})+1)\delta,  
\end{align*}
for all $\omega\in\Sigma_\delta$ and $\eps\in (0,\eps_{\delta,\omega})$.
Since $y\in \calB_{\Lambda/\eps}$ and $\delta>0$ are arbitrary, this proves \eqref{2.7}.
\end{proof}

\subsection{Continuity of the effective velocity}
\label{ss:continuity V}
To end the section, let us prove that the effective velocity $\ol V$ satisfies Assumption \ref{assump2} \ref{item:assump1.1 2}. It suffices to obtain the following result.

\begin{lemma}[Continuity of the effective velocity]\lb{L.4.8}
Let $q_i\in\R$,  and let $a_i$, $b_i$ and $g_i$, for $i=1,2$, satisfy  Assumption \ref{assump5.1}. Further assume that either $g_i$ are strictly positive or $g_1=g_2\equiv 0$. Then  there exits $C$ such that for $\ol V_{q_j}^{(i)}$ given by Corollary \ref{P.2.12} with data $a_i,b_i,g_i$ and $q_j$, if
\[
\|a_1-a_2\|_\infty,\,\|b_1-b_2\|_\infty,\,\|g_1-g_2\|_\infty\leq \delta,\quad |q_1-q_2|\leq \gamma
\]
for some $\delta,\gamma\in (0,1)$, then we have
\begin{equation}
    \label{eq:V continuity}
|\ol V_{q_1}^{(1)}-\ol V_{q_2}^{(2)}|\leq C\delta(1+\min\{|q_1|,|q_2|\})+C\gamma.
\end{equation}
Furthermore, 
there exists $C>0$ such that if $C\delta\leq \gamma$ then, we have
\[
(1+\gamma)^2\ol V_{q_1}^{(1)}\geq \ol V_{(1+\gamma)q_1}^{(2)}.
\]
\end{lemma}
\begin{proof}
We begin by establishing \eqref{eq:V continuity}. 
By the triangle inequality, it suffices to prove the result for $q_1,q_2\geq 0$. 
For $i=1,2$, let $\ol a_i$ be defined as in \eqref{2.10} with $a_i$ in place of $a$, and then
\[
\|\ol a_1-\ol a_2\|_\infty\leq \delta.
\]
In view of the effective equation \eqref{FB} with $x_0=0$ and a fixed $\omega$, consider
\[
\frac{d}{dt}S_{q_i}^{(i)}(t, \omega)= \frac{q_i\ol a_i b_i(S_{q_i}^{(i)}(t,\omega),\omega)}{a_i(S_{q_i}^{(i)}(t,\omega),\omega)}+g_i(S_{q_i}^{(i)}(t,\omega),\omega)
\]
with $S_{q_i}^{(i)}(0,\omega)=0$. Below, we drop $\omega$ from the notations, and denote $f_i(S):=\frac{\ol a_i b_i(S)}{a_i(S)}$. 
Since $a_i,b_i$ are strictly positive and finite, and $g_1$ and $g_2$ are either strictly positive and finite or $g_1=g_2\equiv 0$,  there exists $C>0$ such that for any $S\in\bbR$,
\[
f_2(S)\leq (1+C\delta)f_1(S)\quad\text{and}\quad g_2(S)\leq (1+C\delta)g_1(S).
\]

We will prove that \eqref{eq:V continuity} in two cases: (1) under the additional assumption  $\gamma =0$, and (2) under the additional assumption $\delta=0$. By  the triangle inequality, establishing \eqref{eq:V continuity} in two cases implies that \eqref{eq:V continuity} holds in general.

We begin with case (1): let us assume $\gamma=0$ and then $q:=q_1=q_2$. 
For any $S\in\bbR$, we have
\[
\begin{aligned}
qf_2(S)+g_2(S)&\leq (1+C\delta)(qf_1(S)+g_1(S)).    
\end{aligned}
\]
Set $\tilde S(t):=S_{q}^{(1)}((1+C\delta)t)$ which satisfies
\[
\frac{d}{dt}\tilde S(t) =(1+C\delta)\left[qf_1(\tilde S(t))+g_1(\tilde S(t))\right]\geq qf_2(\tilde S(t))+g_2(\tilde S(t))
\]
which, by the comparison principle, implies that
\[
S_{q}^{(2)}(t)\leq \tilde S(t)=S_{q}^{(1)}((1+C\delta)t).
\]
Thus, Corollary \ref{P.2.12} yields
\[
\ol V_{q}^{(2)}= \lim_{t\to\infty}{S_{q}^{(2)}(t)}/{t}\leq \lim_{t\to\infty}{S_{q}^{(1)}((1+C\delta)t)}/{t}=(1+C\delta)\ol V_{q}^{(1)}.
\]
Because $\ol V_{q}^{(1)}\leq C(1+q)$, we obtain
\[
\ol V_{q}^{(2)}\leq \ol V_{q}^{(1)}+C\delta(1+q).
\]
Reversing the roles of $V^{(1)}_q$ and $V^{(2)}_q$, we therefore obtain that \eqref{eq:V continuity} holds in the case $\gamma =0$. 
In particular, if $g_1\equiv 0$, we have $\ol V_{q}^{(1)}\leq Cq $ and so
\beq\lb{5.19}
\ol V_{q}^{(2)}\leq \ol V_{q}^{(1)}+C\delta q.
\eeq

We proceed with case (2): let us assume $\delta=0$, so that $S_{q_1}^{(1)}=S_{q_1}^{(2)}$, $S_{q_2}^{(1)}=S_{q_2}^{(2)}$,  and $\gamma>0$. We denote $f:=f_1=f_2,g:=g_1=g_2$. 
We break up this case into two subcases: (2a) $g\equiv 0$ and (2b) $g$ strictly positive.  As in case (1), it will suffice to show $\ol V_{q_2}^{(1)}\leq \ol V_{q_1}^{(1)}+C\gamma$, as the reverse inequality follows by reversing the roles of $V_{q_2}^{(1)}$ and $V_{q_1}^{(1)}$.

In case (2a), i.e. if $\delta=0$ and $g\equiv 0$, the ODE satisfied by $S_{q_i}^{(1)}$ becomes,
\[
\frac{d}{dt}S_{q_i}^{(1)}(t, \omega)= q_if \left(S_{q_i}^{(1)}\right).
\]
Uniqueness of solutions to this ODE implies
$
S_{q_i}^{(1)}(t)= S_{1}^{(1)}(q_it)$,
and therefore,  by Lemma \ref{L.2.6}, 
\[
|S_{q_1}^{(1)}(t)-S_{q_2}^{(1)}(t)|=|S_{1}^{(1)}(q_1 t)-S_{1}^{(1)}(q_2 t)|\leq (C_{max}+\|g\|_\infty)|q_1-q_2|t.
\]
This, together with Corollary \ref{P.2.12}, implies that, in this special case,
\[
|\ol V_{q_1}^{(1)}- \ol V_{q_2}^{(1)}|\leq C\gamma.
\]

Let us now take on case (2b): suppose $\delta=0$, $\gamma>0$, and $g$ is strictly positive. The assumptions on $a$, $b$, and $g$ imply, any $S\in\bbR$, 
\[
\begin{aligned}
q_2f(S)+g(S)&\leq q_1f(S)+g(S)+C\gamma.    
\end{aligned}
\]
Let us now further assume $q_1\leq 1$. We use that $g$ is strictly positive to get 
\[
q_2f(S)+g(S)\leq (1+C\gamma)(q_1f(S)+g(S)).
\]
Then, similarly as before,
$\tilde S(t):=S_{q_1}^{(1)}((1+C\gamma)t)$ satisfies
\[
\frac{d}{dt}\tilde S(t) =(1+C\gamma)\left[q_1f(\tilde S(t))+g(\tilde S(t))\right]\geq q_2f(\tilde S(t))+g(\tilde S(t))
\]
which, by the comparison principle, implies that
\[
S_{q_2}^{(1)}(t)\leq \tilde S(t)=S_{q_1}^{(1)}((1+C\gamma)t).
\]
Thus, $\ol V_{q_2}^{(1)}\leq (1+C\gamma)\ol V_{q_1}^{(1)}$. Since $q_1\leq 1$, $\ol V_{q_1}^{(1)}\leq C$ and so
\[
\ol V_{q_2}^{(1)}\leq \ol V_{q_1}^{(1)}+C\gamma
\]
holds, as desired. 
To complete the proof of case (2b), let us now suppose $q_1\geq1$. Since $f$ is strictly positive, we use
\[
q_2f(S)+g(S)\leq (1+C\gamma/q_1)(q_1f(S)+g(S))
\]
to show that $S_{q_2}^{(1)}(t)\leq \tilde S(t)$
where $\tilde S(t):=S_{q_1}^{(1)}((1+C\gamma/q_1)t)$. Since $\ol V_{q_1}^{(1)}\leq C(1+q_1)$,  
this implies that
\begin{equation}
    \label{eq:contV}
\ol V_{q_2}^{(1)}\leq (1+C\gamma/q_1)\ol V_{q_1}^{(1)}\leq \ol V_{q_1}^{(1)}+C\gamma.
\end{equation}
This completes the proof of case (2b), and therefore of \eqref{eq:V continuity}.

Now we show second claim, and we assume $q_1\geq 0$ again. 
Since $g_1\geq 0$, we have
\[
(1+\gamma)\left[q_1f_1(S)+g_1(S)\right]\geq (1+\gamma)q_1f_1(S)+g_1(S). 
\]
By the ODEs and the  comparison principle, 
$
S_{q_1}^{(1)}((1+\gamma) t)\geq S_{(1+\gamma)q_1}^{(1)}(t)$.
This yields 
\beq\lb{eq11}
(1+\gamma)\ol V_{q_1}^{(1)}\geq \ol V_{(1+\gamma)q_1}^{(1)}.
\eeq

If $g_1$ and $g_2$ are strictly positive, due to the first claim and $\gamma\in (0,1)$,  we have
\[
|\ol V_{(1+\gamma)q_1}^{(1)}-\ol V_{(1+\gamma)q_1}^{(2)}|\leq C\delta(1+q_1).
\]
By \eqref{eq11}, we get
\[
(1+\gamma)\ol V_{q_1}^{(1)}\geq \ol V_{(1+\gamma)q_1}^{(1)}\geq \ol V_{(1+\gamma)q_1}^{(2)}-C\delta(1+q_1).
\]
Hence,  to prove the second claim, it suffices to show
\beq\lb{777}
\gamma\ol V_{q_1}^{(1)}\geq C\delta(1+q_1).
\eeq
Note that the comparison principle of ODEs implies for some $c>0$, we have
$\ol V_{q_1}^{(1)}\geq c(1+q_1)$.
Thus, \eqref{777} holds whenever $c\gamma\geq C\delta$.

If $g_1=g_2\equiv 0$, then by \eqref{5.19} and $\gamma\in(0,1)$,
\[
|V_{(1+\gamma)q_1}^{(1)}-V_{(1+\gamma)q_1}^{(2)}|\leq C\delta q_1.
\]
The claim holds the same as before, because $\ol V_{q_1}^{(1)}\geq c q_1$ in this case.
\end{proof}

\section{The Homogenization Result}\lb{S.5}

 This section is devoted to the proof of our main result, Theorem \ref{thm:main}. The main part of the proof is contained in Proposition \ref{P.3.2}. The proof is quite long and divided into several lemmas. In the first subsection, we state this proposition, state the two lemmas that we need for its proof, and use them to prove the proposition. The second subsection is devoted to the proofs of the lemmas. Finally, we prove Theorem \ref{thm:main} in Section \ref{ss:proof of main result}.

\subsection{Main proposition}
\label{ss:main prop}
Let $p_\eps$ be the solution to \eqref{1.1} with $\eps\in (0,1)$, and let $\Omega_\eps:=\Omega_{p_\eps}$ be the positivity set of $p_\eps$. Then we take half-relaxed limit of the sets $\Omega_\eps$:
\[
\Omega^*(t,\omega):=\limsup_{\eps\to 0,s\to t}\Omega_\eps(s)=\bigcap_{c_1,c_2\to 0}\bigcup_{|s-t|<c_1,\eps<c_2}\Omega_\eps(s),
\]
and
\[
\Omega_*(t,\omega):=\liminf_{\eps\to 0,s\to t}\Omega_\eps(s)=\bigcup_{c_1,c_2\to 0}\bigcap_{|s-t|<c_1,\eps<c_2}\Omega_\eps(s).
\]

The goal is to show that $\Omega^*$ and $\Omega_*$ are, respectively,  a viscosity subflow and superflow of \eqref{1.3}. For this purpose, we denote by $p^*$ the associated function of $\Omega^*$ with respect to the elliptic equation in \eqref{1.3}, meaning that $p^*(\cdot,t)$ is the solution to  
\beq\lb{998}
-\partial_x(\overline A(x) \partial_x p)=\overline F(x)
\eeq
in the domain  $\Omega^*(t)$ with $0$ boundary data.
Similarly, we use $p_*$ to denote the associated function of $\Omega_*$ with respect to \eqref{1.3}.
By Lemma \ref{L.def}, if $\Omega_*$ (resp. $\Omega^*$) is a viscosity superflow (resp. subflow) to \eqref{1.3}, then $p_*$ (resp. $p^*$) is a viscosity supersolution (resp. subsolution).

\begin{proposition}\lb{P.3.2}
$\Omega^* $ is a viscosity subflow and $\Omega_* $ is a viscosity superflow of \eqref{1.3} with $\ol V$ defined in \eqref{3.11}, and with $\ol A,\ol F$ from Lemma \ref{L.9.2}.
\end{proposition}

We will give the proof that  $\Omega_*$ is a superflow; the proof for $\Omega^*$ is analogous. Given the complexity of the proof, we divide it into lemmas. We will state the two lemmas that we need, then use them to prove the proposition, and then prove the lemmas. 

In the first lemma, we assume that a smooth function $\varphi$  touches $p_*$ from below at some $(x_0, t_0)\in \Gamma_{p_*}$, and yet the desired inequality
\[
\varphi_t(x_0,t_0)\geq \ol V(x_0,\varphi_x(x_0,t_0))|\varphi_x(x_0,t_0)|
\] 
fails. We will use this to reduce \eqref{1.1} to a problem that is independent of the slow variable.

Before stating the lemma, we introduce some notation: 
\[
N(R_0):=\{(x,t)\mid t_0-R_0<t\leq t_0, |x-x_0|\leq R_0\},
\]
\beq\label{defab}
a(z,\omega)=A(x_0,z,\omega),\quad b(z,\omega)=B(x_0,z,\omega)\quad\text{and}\quad g(z,\omega)=G(x_0,z,\omega),
\eeq
and, for $q\geq 0$, 
\beq\lb{3.11}
\begin{aligned}
\overline V(x_0,q)&:=\ol V_q(A(x_0,\cdot,\cdot),B(x_0,\cdot,\cdot), G(x_0,\cdot,\cdot)),\\
\overline V(x_0,-q)&:=\ol V_q(A(x_0,-\cdot,\cdot),B(x_0,-\cdot,\cdot),G(x_0,\cdot,\cdot)),
\end{aligned}
\eeq
where $\ol V_q(a,b)$ is given in \eqref{speed}.
Furthermore, we will often drop $\omega$ from the notations.

\begin{lemma}\label{lem for prop 1}
Suppose $(x_0,t_0)$ is a right-hand side free boundary point of $p_*$ and  $\varphi(x,t)$ is a smooth function that touches $p_* $ at $(x_0,t_0)$ locally in 
$N({R_0})$ from below. 
Let us further assume that
\begin{align}
&|\varphi_x(x_0,t_0)|\neq 0 \quad 
\text{ and} \lb{3.31}\\
&\varphi_t(x_0,t_0)+\delta_0< \ol V(x_0,\varphi_x(x_0,t_0))|\varphi_x(x_0,t_0)| \lb{4.10}
\end{align}
for some $\delta_0>0$. 

There exist $R\in (0,1)$, a subsequence $\ep_k\rightarrow 0$, a smooth function $\tilde\varphi(x,t)$,  solutions $p_{2,\eps_k}$ to
\begin{equation}\lb{4.12}
\left\{
\begin{aligned}
   & -\partial_x(a(\eps_k^{-1}x) \partial_x p_{2,\eps_k})=0 &&\,\text{ in }N(R)\cap\{p_{2,\eps_k}>0\},\\
   &\partial_t p_{2,\eps_k}=b(\eps_k^{-1}x)| \partial_x p_{2,\eps_k}|^2+g(\eps_k^{-1}x)| \partial_x p_{2,\eps_k}|
 &&\,\text{ on }N(R)\cap\partial\{p_{2,\eps_k}>0\},
\end{aligned}   
\right.
\end{equation}
with $\{x_{\eps_k}(t)\}=\partial\{p_{2,\eps_k}(\cdot,t)>0\}\cap (x_0-R,x_0+R)$ and $\lim_{k\to\infty}x_{\eps_k}(t_0)=x_0$, and a function $p_{2,*}\geq 0$ that is continuous in space and lower semi-continuous in time,
such that the following hold:
\begin{enumerate}[(i)]
\item \label{item:p2ep p*} 
$\displaystyle 
\liminf_{k\to\infty}p_{2,\eps_k}(x,t)\geq p_{2,*}(x,t)$.

\item $\tilde\varphi(x_0,t_0)=0$ and
$
|\tilde\varphi_x(x_0,t_0)|\neq 0$. 
\item $\tilde \varphi(x,t)$ touches $p_{2,*} $ at $(x_0,t_0)$  locally in $N(R)$ from below;

\item \label{item:V ineq tilde} For some $\tilde\delta_0>0$,
\[
\tilde \varphi_t(x_0,t_0)+\tilde\delta_0< \ol V(x_0,\tilde\varphi_x(x_0,t_0))|\tilde\varphi_x(x_0,t_0)|.
\]
\end{enumerate}
\end{lemma}

The following lemma is at the heart of our proof. We use a homogenization argument to obtain a contradiction with the fact that $\lim_{\ep\rightarrow 0}x_{\ep}(t_0)=x_0$.  Roughly speaking, we suppose that the velocity is ``too fast," i.e. \eqref{4.10} holds. We also suppose that, for \emph{all} times $t\in (t_-, t_0)$, and at some point $x_3<x_0$, the $p_{\ep_k}$ are strictly bigger than a linear profile $P(x,t)$ that moves with speed \emph{bigger} than the homogenized velocity. Then we conclude that, for all $\ep_k$ small enough, the free boundary of the $p_{\ep_k}$ at time $t_0$ must be to the right of $x_0$: in other words, the profile pushes $x_{\ep_k}(t)$ too far to the right. This will yield the desired contradiction in the proof of Proposition \ref{P.3.2}. We point out that Proposition \ref{P.2.13} (uniform convergence of $S^{x_0}_{q,\ep}$) is also an essential element of the proof of Lemma \ref{lem:prop lemma 2}. 

\begin{lemma}\label{lem:prop lemma 2}
Suppose $R\in (0,1)$ and $p_{\ep_k}$ satisfies \eqref{4.12} in $N(R)\cap\{p_{\ep_k}>0\}$. Suppose $x_{\ep_k}(t)$ is a right-hand side free boundary point of $p_{\ep_k}$. Let the profile $P_2$ be given by,
\[
P_2(x,t) = q_2(-r_2(t_0-t)+x_0+\delta_2 -x), 
\]
where $r_2$, $q_2$, $\delta_2$ are positive and such that
\begin{equation}
    \label{item:4.1'} r_2<\ol V(x_0, q_2).
\end{equation}
Furthermore, suppose that the following hold for some $t_-\in (t_0-R, t_0)$: 
\begin{enumerate}[(i)]
    \item \label{item:4.0''} 
    For $x_2:= -r_2(t_0-t_-)+x_0+\delta_2$, we have $x_2\in \Gamma_{P_2}(t_-)$ and  $x_2\in (x_0-R, x_0)$.
    \item \label{item:4.2} For all $k$ sufficiently large, we have $x_{\ep_k}(t_-)>x_2$. 
    \item \label{item:4.7}  There exists $x_3\in (x_0-R, x_2)$ such that, for all $k$ large and for all $t\in (t_-, t_0)$,
    \beq
    \label{eq:item:4.7}
    P_2(x_3, t)<p_{\ep_k}(x_3, t)-q_2\delta_2/2.
    \eeq
\end{enumerate}
Then we have, 
\begin{equation}
    \label{eq:lem 2 conclusion}
    x_\ep(t_0)\geq x_0+\delta_2/3.
\end{equation}
\end{lemma}

\begin{proof}[Proof of Proposition \ref{P.3.2}]
Let us prove the claim for $\Omega_*$. 
Take an arbitrary right-hand side free boundary point $(x_0,t_0)$ and let $\varphi(x,t)$ be a smooth function that touches $p_* $ at $(x_0,t_0)$ locally in $\{(x,t)\mid p_* (x,t)>0,t\leq t_0\}$ from below. Namely, $\varphi(x_0,t_0)=0$ and $\varphi(x,t)< p_* (x,t)$ in $N(R_0)\cap\{p_* (x,t)>0\}$ for some $R_0>0$. 
Let us further assume that
\eqref{3.31}
 holds. 
Since $(x_0,t_0)$ is a right-hand side free boundary point,  the free boundary is propagating to the right and $\varphi_x(x_0,t_0)<0$. Since the support of $p_* $ is  non-decreasing, we have $\varphi_t(x_0,t_0)\geq 0$. 
In view of Definition \ref{D.1.3}, to show that $p_* $ is a supersolution, it suffices to show,
\[
\varphi_t(x_0,t_0)\geq \ol V(x_0,\varphi_x(x_0,t_0))|\varphi_x(x_0,t_0)|,
\]
where $\ol V$ is given in \eqref{3.11}.
Thus, we assume for contradiction that this fails, so that \eqref{3.31} and \eqref{4.10} hold 
for some $\delta_0>0$ which is independent of $R_0$.

We are now exactly in the setting of Lemma \ref{lem for prop 1}. Applying the lemma,  we find that there exist  $p_{2,\eps_k}$ satisfying
the equation \eqref{4.12}, $R>0$, and a smooth function $\tilde{\varphi}$, such that items \ref{item:p2ep p*}-\ref{item:V ineq tilde} of Lemma \ref{lem for prop 1} hold. 
Therefore, without loss of generality, we can simply assume that $p_{\eps_k}=p_{2,\eps_k}$ satisfies \eqref{4.12}, 
\begin{equation}
    \label{eq:liminf pep}
    \liminf_{k\to\infty}p_{\eps_k} \geq  p_* , 
    \end{equation}
    and its free boundary $(x_{\eps_k}(t),t)$ satisfies
\beq\lb{5.5}
x_{\eps_k}(t_0)\to x_0 \quad\text{as }k\to \infty.
\eeq
For the remainder of the proof of this proposition, we  write $(p_*,\varphi,\delta_0)$ instead of $(p_{2,*},\tilde\varphi,\tilde\delta_0)$.

 Let $q_0:=-\partial_x\varphi(x_0,t_0)>0$ and $r_0:=\partial_t\varphi(x_0,t_0)/q_0>0$. Since $\varphi$ touches $p_* $ from below and \eqref{4.10} holds, by taking $\delta_1$ sufficiently small, we can guarantee that for 
\[
q_1:=q_0-\delta_1>0\quad\text{and}\quad r_1:=r_0+\delta_1,
\]
there exists $R_1=R_1(\delta_1)\in (0,R)$ such that 
\begin{align}
    \lb{4.6}
r_1&<\ol V(x_0,q_1) \text{ and}\\
\lb{4.3}
P_1(x,t)&:=q_1(-r_1(t_0-t)+(x_0-x))_+<p_* (x,t)
\end{align}
for all $(x,t)\in \left(\overline{N(R_1)\cap \{P_1>0\}}\right)\setminus\{(x_0,t_0)\}$.

We remark that, since $r_1$ is the speed of  the profile and $q_1$ is its slope, the inequality  \eqref{4.6} means that the profile $P$  moves with speed smaller than $\overline{V}(x_0, q_1)$. 
Furthermore, then inequality \eqref{4.3} implies that the free boundary of $p_* (\cdot,t)$ in $N(R)$, which we denote by $x_{* }(t)$, lies to the right of the free boundary of $P_1(t)$.  Since $\Gamma_{P_1}(t)=\{-r_1(t_0-t)+x_0\}$, we obtain,
\beq
\label{eq:x* Gamma P1}
x_* (t)>-r_1(t_0-t)+x_0 \text{ for all }t\in (t_0-R,t_0).
\eeq

Now we will define a perturbed slope and speed, $q_2$ and $r_2$, so that the resulting  profile, which we will call $P_2$ and define in \eqref{eq:defP2} below,  satisfies the hypotheses of Lemma \ref{lem:prop lemma 2}. To this end, first fix some $t_{-}\in (t_0-R_1/r_1,t_0)$; the requirement  $t_{-}>t_0-R_1/r_1$ ensures   that
\beq
\label{P1x1t-}
x_1:=-r_1(t_0-t_{-})+x_0\in \Gamma_{P_1}(t_{-})\quad\text{satisfies}\quad x_1>x_0-R_1.
\eeq
In particular, since \eqref{eq:x* Gamma P1} holds for $t=t_-$, we find,
\beq\label{eq:x-x*t-}
x_*(t_-)>x_1.
\eeq
Next, for $\delta_2\gg\eps>0$ small and to be determined, let $q_2$ be the unique positive root of
\[
q\mapsto 2\delta_2 q^2+r_1q_1(t_0-t_-)q -r_1q_1^2(t_0-t_-),
\]
and define
\beq\lb{4.4}
r_2:=r_1q_1/q_2.
\eeq
These definitions, together with an elementary calculation, yield
\begin{align}
\label{6.15.1}
r_1q_1(q_1-q_2)(t_0-t_{-})&=2\delta_2q_2^2\\
\label{6.15.2}
r_2&>r_1,\\
q_2&<q_1, \nonumber 
\end{align}
 as well as $\lim_{\delta_2\to 0}q_2=q_1$, and $\lim_{\delta_2\to 0}r_2=r_1$.
 As a consequence,  \eqref{4.6} and the continuity of $\ol V$ imply that
\beq\lb{4.1'}
r_2<\ol V(x_0, q_2)
\eeq
holds for $\delta_2>0$ sufficiently small. 
Next, for $\lambda>0$ to be chosen shortly, define 
\beq
\label{6.16.5}
x_2:=-r_2(t_0-t_{-})+x_0+\delta_2,\quad x_{3}:=x_2-\lambda.
\eeq
We note $\lim_{\delta_2\rightarrow 0, \lambda\rightarrow 0}x_3=x_0$. Since we have $x_1>x_0-R_1$ --- see \eqref{P1x1t-} --- we can  ensure 
\[
x_3>x_0-R_1
\]
holds by taking $\delta_2>0$ and $\lambda>0$ sufficiently small. Finally, let us take $\delta_2>0$ small enough to guarantee that $q_2>\frac{3}{4}q_1$ holds. As a consequence, we find,  
\begin{align}
x_1-x_2&=(r_2-r_1)(t_0-t_-) \delta_2 \quad &\text{ by \eqref{6.16.5}}\nonumber\\
& = r_1\left(\frac{q_1}{q_2}-1\right) (t_0-t_-) &\text{ by \eqref{4.4}} \nonumber\\
&=\frac{2\delta_2q_2}{q_1}-\delta_2 \quad &\text{ by \eqref{6.15.1}} \nonumber\\
&>\frac{1}{2}\delta_2. \lb{4.20}
\end{align}

Now we fix $\delta_2$ and $\lambda$; therefore, $r_2$, $q_2$, $x_2$ and $x_3$ are also now fixed, and all these constants are independent of $\omega$ and $\ep$. 
Finally, we define the perturbed profile 
\beq
\label{eq:defP2}
P_2(x,t):=q_2(-r_2(t_0-t)+(x_0+\delta_2-x))_+.
\eeq
See Figure \ref{figure Prop 6.2} for an illustration. It is direct to check that
\beq\lb{4.0''}
x_0+\delta_2\in \Gamma_{P_2}(t_0)\quad \text{and}\quad
x_2\in \Gamma_{P_2}(t_{-}).
\eeq
Furthermore, \eqref{eq:x-x*t-} and  \eqref{4.20} imply,
\beq\lb{4.1}
x_2+\frac12\delta_2<x_1<x_* (t_{-}).
\eeq
In particular, we note,
\beq
\label{x3x0}
x_3<x_2<x_1<x_0,
\eeq
 where we recall $x_3$ was defined in \eqref{6.16.5}.

From \eqref{x3x0} and \eqref{4.0''} we find $P_2(x_3, t)>0$   for all $t\in (t_-, t_0)$; similarly, \eqref{x3x0} and \eqref{P1x1t-} imply $P_1(x_3, t)>0$ for all $t\in (t_-, t_0)$. Using $P_2(x_3, t)>0$, followed by \eqref{4.4}, \eqref{6.15.2}, the definition of $P_1$, the positivity of $P_1(x_3, t)$, and then \eqref{4.3}, we find, for all $t\in (t_{-},t_0)$,
\begin{align*}
P_2( x_3,t)&= q_2(-r_2(t_0-t)+x_0+\delta_2- x_3)\\
&\leq  q_1(-r_1(t_0-t)+x_0- x_3)-q_2\delta_2= P_1(x_3, t)-q_2\delta_2<p_* ( x_3,t)-q_2\delta_2.    
\end{align*}

Recall that \eqref{eq:liminf pep} holds,  and $x_3=-r_2(t_0-t_{-})+x_0+\delta_2-\lambda$. Thus, for $t\in (t_{-},t_0)$,
\beq\lb{4.7}
q_2(r_2(t-t_{-})+\lambda)=P_2(x_{3},t)<p_{\eps_k}(x_{3},t)-q_2\delta_2/2\quad \text{ when $k$ is large. }
\eeq
It follows from \eqref{4.1} that,
for $k$ sufficiently large,
\beq\lb{4.2}
 x_{\eps_k}(t_{-})>x_2\in \Gamma_{P_2}(t_{-}).
\eeq
where $x_{\eps_k}(t_{-})$ is such that $x_{\eps_k}(t_{-})\in \Gamma_{p_{\eps_k}}(t_-)$ and $
\lim_{k\to\infty}\lim_{t\to t_0}x_{\eps_k}(t)\to x_0$.

\begin{figure}
\label{figure Prop 6.2}
    \centering
    \includegraphics[scale=.5]{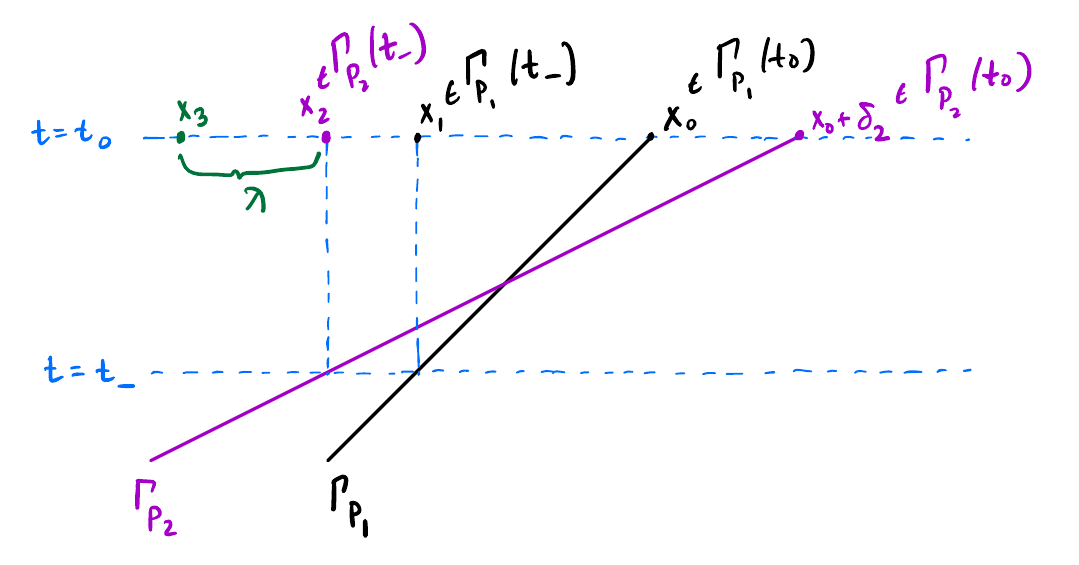}
    \caption{Illustration for the proof of Proposition \ref{P.3.2}.}
\end{figure}

We are now in a position to apply Lemma \ref{lem:prop lemma 2}. We note that \eqref{item:4.1'} holds due to \eqref{4.1'}. Further, items \ref{item:4.0''}, \ref{item:4.2}, and \ref{item:4.7} hold due to \eqref{4.0''}, \eqref{4.2}, and \eqref{4.7}, respectively. 
The conclusion of Lemma \ref{lem:prop lemma 2}, namely \eqref{eq:lem 2 conclusion}, contradicts \eqref{5.5}. 
Thus, we have reached the desired contradiction and proved that $p_* $ is a supersolution.
\end{proof}

\subsection{Proofs of lemmas}
\label{ss:proof of lemmas}
We begin with:

\begin{proof}[Proof of Lemma \ref{lem for prop 1}]
Throughout the proof, we fix $\omega$. Recall that the $p_\eps$ are  viscosity solutions to \eqref{1.1},  and therefore they  are also the associated functions to the corresponding viscosity flow.

Let us denote the free boundary of $p_* $ in $N(R_0)$ as $(x_* (t),t)$ and the free boundary of $p_{\eps}$ as $(x_{\eps}(t),t)$. We have $x_* (t)\leq x_0$ for $t\in (t_0-R_0,t_0)$. The definition of $\Omega_*$ implies that,   for the fixed $\omega$, 
\beq\lb{999}
x_{\eps_k}(t_k)\to x_* (t_0)\quad\text{ as } k\to\infty
\eeq
along a sequence of $\eps_k\to 0$ and $t_k\to t_0$.
Lemma \ref{L.9.2} (interior homogenization) yields that $\liminf_{k\to\infty}p_{\eps_k}(\cdot,t)$ is a  non-negative supersolution to \eqref{7.2} or \eqref{998} in its positivity set. The stability of viscosity solutions yields that 
\[
\liminf_{k\to\infty,\,(x_k,t_k)\to (x,t)}p_{\eps_k}(x_k,t_k)
\]
is again a non-negative supersolution in its positivity set; this positivity set coincides with $\Omega_*$ by the definition of $\Omega_*$. Since $p_*$ solves the same equation in $\Omega_*$ with $0$ boundary data, we have 
\beq\lb{990}
\liminf_{k\to\infty,\,(x_k,t_k)\to (x,t)}p_{\eps_k}(x_k,t_k)\geq p_*(x,t).
\eeq

By replacing $p_{\eps_k}(x,t)$ by $p_{\eps_k}(x,t-t_k+t_0)$, we can assume that $t_k=t_0$.

Let $R\leq R_0$ be small enough such that for all large $k$ we have $p_{\eps_k}(x,t)>0$  in $N_R$. 
We have that $p_{\eps_k}$ satisfies the following equation
\begin{equation}
\label{eq:pepk}
\left\{
\begin{aligned}
    -\partial_x(A(x,\eps_k^{-1}x ) \partial_x  p)&=F(x,\eps_k^{-1}x ) &&\quad\text{ in }N(R)\cap\{p_{\eps_k}>0\}:=N_{R,k},\\
   p(x,t)&=p_{\eps_k}(x,t )  &&\quad\text{ if }x\in\{x_0-R,x_{\eps_k}(t)\}.
\end{aligned}   
\right.
\end{equation}
A direct computation thus yields that    $p_{\eps_k}$ is given by the following formula in $N_{R,k}$: 
\begin{equation}
\label{eq:expression pep}
p_{\eps_k}(x,t)=   \int_x^{x_{\eps_k}(t)}\frac{C_t}{A(y,\eps_k^{-1}y)}-\frac{\int_{y}^{x_{\eps_k}(t)}F(z,\eps_k^{-1}z)dz}{A(y,\eps_k^{-1}y)}dy
\end{equation}
where  $C_t$ is determined by the boundary value $p_{\eps_k}(x_0-R,t)$ and is given by,
\beq
\label{eq:expression C}
C_t=  \left(\int_{x_0-R}^{x_{\eps_k}(t)}\frac{dy}{A(y,\eps_k^{-1}y)}\right)^{-1} \left(\int_{x_0-R}^{x_{\eps_k}(t)}\frac{\int_{y}^{x_{\eps_k}(t)}F(z,\eps_k^{-1}z)dz}{A(y,\eps_k^{-1}y)}dy +p_{\eps_k}(x_0-R,t)\right).
\eeq
We will use the following facts.
 Since 
 $x_{\eps_k}(t_0)\to x_*(t_0)$, we have that $x_{\eps_k}(t)-x_0+R\leq 2R$ for $t\leq t_0$ and for $k$ sufficiently large.  Since $p_{\eps_k}(\cdot,t)$ is uniformly Lipschitz (locally uniformly in time), $p_{\eps_k}(x_0-R,t)\leq CR$ for all $t\in (t_0-R,t_0)$. 

Further, let us fix $z<x_0$ such that $p_*(z,t)\geq \tilde{c}>0$ for all $t\in [t_0-R, t_0]$, where $\tilde{c}>0$ depends only on $p_*$, and thus is independent of $k$ and $\omega$). Thus, according to \eqref{990}, we have that, for all $k$ large enough,
$p_{\ep,k}(x,t)\geq \frac{c}{2}$ for all $t\in [t_0-R, t_0]$. Since $C(x-x_{\ep_k}(t))$, where $C$ depends only on the assumptions,  is a subsolution of the elliptic equation in \eqref{eq:pepk}  for   each fixed  $t\in [t_0-R, t_0]$, we therefore find that
\beq\lb{334}
p_{\eps_k}(x,t)\geq c(x-x_{\eps_k}(t))_+
\eeq
holds  in $N_{R,k}$, for some $c>0$ independent of $k$ and $\omega$. Similarly, we also obtain, for $(x,t)\in N_{R,*}$,
\beq\lb{334.5}
p_{*}(x,t)\geq c(x-x_{\eps_k}(t))_+.
\eeq

Next,  let $p_{1,{\eps_k}}(\cdot,t)$ be the unique solution to the homogeneous equation
\begin{equation*}
\left\{
\begin{aligned}
   & 
   -\partial_x(a(\eps_k^{-1}x ) \partial_x  p_{1,{\eps_k}})=0 &&\quad\text{ in }N_{R,k},\\
  & p_{1,{\eps_k}}(x,t)=p_{\eps_k}(x,t )  &&\quad\text{ if }x\in\{x_0-R,x_{\eps_k}(t)\},
\end{aligned}   
\right.
\end{equation*}
where we've used the notation \eqref{defab}. Note that for $x\in [x_0-R,x_0+R]$, 
\[
|A(x,z)-a(z)|\leq C|x-x_0|\leq CR.
\]
Upon writing an expression analogous to \eqref{eq:expression pep}-\eqref{eq:expression C} but for $p_{1,\ep_k}$, and using the previous estimate and the fact that $F$ is uniformly bounded, an explicit calculation yields that,  for all $(x,t)\in N_{R,k}$,
\begin{align}
\lb{994}
|p_{\eps_k}(x,t)-p_{1,{\eps_k}}(x,t)|&\leq C(x_{\eps_k}(t)-x)_+ R \quad \text{ and }\\
\lb{991}
|\partial_x p_{\eps_k}(x_{\eps_k}(t),t)-\partial_xp_{1,{\eps_k}}(x_{\eps_k}(t),t)|&\leq CR,
\end{align}
where $C$ is independent of $R\leq R_0$.

However, here $p_{1,{\eps_k}}$ does not satisfy the free boundary condition at $x=x_{{\eps_k}}(t)$. We further make the following adjustment.
Let $\gamma_{\eps_k}(t )$, which depends  on $\omega$, be such that
\beq\lb{992}
\begin{aligned}
&\quad\, (1+\gamma_{\eps_k}(t ) )b(\eps_k^{-1}x_{\eps_k}(t))| \partial_x p_{1,{\eps_k}}(x_{\eps_k}(t),t)| +g(\eps_k^{-1}x_{\eps_k}(t))\\
&=B(x_{\eps_k}(t),\eps_k^{-1}x_{\eps_k}(t))| \partial_x p_{{\eps_k}}(x_{\eps_k}(t),t)| +G(x_{\eps_k}(t),\eps_k^{-1}x_{\eps_k}(t)),    
\end{aligned}
\eeq
and define 
\[
p_{2,{\eps_k}}(x,t):=(1+\gamma_\ep(t))p_{1,{\eps_k}}(x,t).
\]
Since $\partial_x p_{1,{\eps_k}}(x_{\eps_k}(t),t)$ and $\partial_x p_{{\eps_k}}(x_{\eps_k}(t),t)$ are strictly positive for all  $k$ and $\omega$ by \eqref{334} and \eqref{991}, and due to the regularity assumptions on the coefficients, we have $|\gamma_{\eps_k}(t)|\leq CR$ uniformly for all $k$ and $\omega$.
And, $p_{2,{\eps_k}}$ satisfies the same elliptic equation as $p_{1,\eps_k}$:
\[
 -\partial_x(a(\eps_k^{-1}x ) \partial_x  p_{2,{\eps_k}})=0\quad\text{ in its positivity set in } N_{R,k}.
\]
Since $p_{\eps_k}$ are solutions to \eqref{1.1}, and $p_{\eps_k}$ and $p_{2,\eps_k}$ have the same support, then \eqref{992} yields that $p_{2,\eps_k}$ satisfies the free boundary condition in the viscosity sense as well:
\[
 \partial_t p_{2,{\eps_k}}=b(\eps_k^{-1}x)| \partial_x p_{2,{\eps_k}}|^2+g(\eps_k^{-1}x)| \partial_x p_{2,{\eps_k}}| \quad\text{ on the free boundary }x=x_{\eps_k}(t).
\]
Moreover, from \eqref{994}, 
we have, for all $(x,t)\in N_{R,k}$,
\beq\lb{333}
|p_{\eps_k}(x,t)-p_{2,{\eps_k}}(x,t)|\leq C(x_{\eps_k}(t)-x)_+R.
\eeq

Now, for some $C_1>0$ to be determined, we define 
\[
p_{2,*}(x,t):=  \frac{(1-C_1R)p_{*}(x_0-R,t )}{(x_*(t)-x_0+R)}    (x_*(t)-x)_+,
\]
which is  the unique solution to
\begin{equation*}
\left\{
\begin{aligned}
   & - \partial_x^2\,  p_{2,*}=0 \qquad\text{ in }N(R)\cap\{p_*>0\}=:N_{R,*},\\
   &p_{2,*}(x_0-R,t)=(1-C_1R) p_{*}(x_0-R,t )  \,\text{ and }\, p_{2,*}(x_*(t),t)=0.
\end{aligned}   
\right.
\end{equation*}
Recall that $p_*$ satisfies \eqref{998} in the domain  $N_{R,*}$. Thus, $p_*$ can be written explicitly in terms of $\bar A$ and $\bar F$ as in  \eqref{eq:expression pep}-\eqref{eq:expression C}. This expression for $p_*$, together with the definition of $p_{2,*}$, the Lipschitz continuity of $A$ and the boundedness of $\bar F$,  yields, as before,
\begin{align}
\lb{335}
|(1-C_1R)^{-1}p_{2,*}(x,t)-p_{*}(x,t)|&\leq C(x_{*}(t)-x)_+R
\end{align}
for $(x,t)\in N_{R,*}$. 
Using the bound \eqref{334.5}, followed by  \eqref{335}, yields, 
\beq\lb{331} 
(1-C_1R/2)p_*(x,t)\geq (1-C_1R)p_*(x,t)+cC_1R(x_*(t)-x)_+\geq p_{2,*}(x,t)\quad\text{ in }N_{R,*},
\eeq
when $C_1$ is sufficiently large but independent of $R$.

Using \eqref{333},  \eqref{334}, \eqref{990}, and then \eqref{331}, 
we find that if $C_1$ is sufficiently large independent of $R$, then for $(x,t)\in N_{R,*}$,
\beq\lb{995}
\begin{aligned}
\liminf_{k\to\infty}p_{2,\eps_k}(x,t)&\geq \liminf_{k\to\infty} \left[p_{\eps_k}(x,t)- CR(x_{\eps_k}(t)-x)_+\right]\\
&\geq (1-CR)\liminf_{k\to\infty}p_{\eps_k}(x,t)\\
&\geq (1-CR)p_*(x,t)\geq p_{2,*}(x,t).    
\end{aligned}
\eeq

Next, note that \eqref{335} implies, for $(x,t)\in N_{R, *}$,
\begin{align*}
    p_{2,*}(x,t)&\geq (1-C_1R)p_*(x,t)-(1-C_1R)CR(x_*(t)-x)_+.
\end{align*}
Since $x_*(t)\leq x_0$ for $t\leq t_0$, we have, for $x\leq x_0$ and $t\leq t_0$, 
\[
(x_*(t)-x)_+\leq (x_0-x)_+ = (x_0-x). 
\]
Combining this with the previous inequality yields that, in $N(R)\cap\{x\leq x_0\}$,
\[
p_{2,*}(x,t)\geq (1-C_1R)p_*(x,t)-CR(x_0-x).
\]
We use that $\varphi$ touches $p_*$ from below at $(x_0,t_0)$ locally in $N(R)$ to bound the right-hand side of the previous line from below and find, for $(x,t)$ near  $(x_0, t_0)$ and with $t\leq t_0$,
\[
p_{2,*}(x,t)\geq (1-C_1R)\varphi(x,t)-CR(x_0-x).
\]
The previous line and the fact that $\varphi$ is smooth imply that there  exists a positive constant $C'$, which is independent of $R$,  such that,  for $(x,t)$ near  $(x_0, t_0)$ and with $t\leq t_0$,
\begin{align*}
p_{2,*}(x,t)
&\geq \varphi(x,t)-C'R(x_0-x+t_0-t)    
\end{align*}
 Since $p_{2,*}(x_0,t_0)=p_*(x_0,t_0)=\varphi(x_0,t_0)=\tilde \varphi(x_0,t_0)=0$ and $x_0$ is a right-hand side free boundary, 
we conclude that
\[
\tilde \varphi(x,t)=\varphi(x,t)-C'R(x_0-x+t_0-t)
\]
touches $p_{2,* }$ from below at $(x_0,t_0)$ locally in  $N(R)$.

The construction implies that $\tilde\varphi_t(x_0,t_0)>0$.
Note that $\delta_0$ in \eqref{4.10} is independent of $R$. Thus, by taking $R$ to be small, \eqref{4.10} implies
\[
\tilde \varphi_t(x_0,t_0)+\tilde\delta_0< \ol V(x_0,\tilde\varphi_x(x_0,t_0))|\tilde\varphi_x(x_0,t_0)|
\]
with $\tilde\delta_0:=\delta_0/2$.
\end{proof}

We move on to:
\begin{proof}[Proof of Lemma \ref{lem:prop lemma 2}]
We begin by making the  rescaling
\begin{align*}
    v_{\eps_k}(x,t)&:=\eps_k^{-1}p_{\eps_k}\left(\eps_k(x-x_2)+x_2,\eps_k(t-t_{-})+t_{-}\right) \quad \text{and}\\
    X_k(t)&:= \eps_k^{-1}\left[x_{\eps_k}\left(\eps_k \left(t-t_{-}\right)+t_{-}\right)-x_2\right]+x_2.
\end{align*}
Using that  $p_{\ep_k}$ satisfies \eqref{4.12}, we find that  
$v_{\ep_k}(\cdot, t, \omega)$ satisfies, 
for all $t\in \left(t_-+\frac{t_0-t_--R}{\ep_k}, t_-+\frac{t_0-t_-}{\ep_k}\right)$,
\begin{equation}
\label{eq:vepk interior}
\left\{
\begin{aligned}
   &-\partial_x(a(x-x_2+x_2/\eps_k,\omega) \partial_x v_{\eps_k}(x,t,\omega))=0\quad \text{ for } x\in (x_2-\tilde{\lambda}/\eps_k,X_k(t)),\\
    &v_{\eps_k}(X_k{(t)},t,\omega)=0,
\end{aligned}   
\right.
\end{equation}
where $\tilde{\lambda} :=x_2-(x_0-R)$. We note that  assumption \ref{item:4.0''} of this lemma says $x_2\in (x_0-R, x_0)$, which implies $\tilde{\lambda}>0$.

Next, recalling that $p_{\ep_k}$ satisfies \eqref{4.12} and utilizing the notation introduced in \eqref{defab}, we obtain, for 
\[
y_{k}:=-x_2+\eps_k^{-1}x_2,
\]
that the free boundary condition in \eqref{4.12} becomes
\beq\lb{4.9'}
\begin{aligned}
\frac{(v_{\eps_k})_t}{|(v_{\eps_k})_x|} (X_k(t),t)&=b\left(X_k(t)-x_2+\eps_k^{-1}x_2,\omega\right)|\partial_x v_{\eps_k}(X_{k}(t),t)|+g(X_k(t)-x_2+\eps_k^{-1}x_2,\omega)\\
&=b(X_k(t),\tau_{y_k}\omega)|\partial_x v_{\eps_k}(X_{k}(t),t)|    +g(X_k(t),\tau_{y_k}\omega)
\end{aligned}
\eeq
in the viscosity sense. Here $\tau_y$ is the measure-preserving transformation group.

Note that the support of \( v_{\varepsilon_k}(\cdot, t) \) is a union of intervals, each of which expands continuously and monotonically in time. Therefore, by definition, \( v_{\varepsilon_k} \) is continuous for almost every time \( t \). (Discontinuities in \( v_{\varepsilon_k}(t, \cdot) \) may occur only at times when two disjoint intervals in its support merge into one.) Moreover, from the explicit formula, \( \partial_x v_{\varepsilon_k}(X_k(t), t) \) is continuous for almost every \( t \).
This implies that
\[
(v_{\eps_k})_t(X_k(t),t)=b(X_k(t),\tau_{y_k}\omega)|\partial_x v_{\eps_k}(X_{k}(t),t)|^2    +g(X_k(t),\tau_{y_k}\omega)|(v_{\eps_k})_x (X_k(t),t)|
\]
is continuous for almost all $t\in \left(t_-+\frac{t_0-t_--R}{\ep_k}, t_-+\frac{t_0-t_-}{\ep_k}\right)$. For those $t$, \eqref{4.9'} holds in the classical sense. 
Therefore, $X_k(t)$ is Lipschitz continuous and  satisfies, for $t\in \left(t_-+\frac{t_0-t_--R}{\ep_k}, t_-+\frac{t_0-t_-}{\ep_k}\right)$,
\beq\lb{4.9}
\begin{aligned}
\frac{d}{dt}X_k(t)=b(X_k(t),\tau_{y_k}\omega)|\partial_x v_{\eps_k}(X_{k}(t),t)|    +g(X_k(t),\tau_{y_k}\omega)
\end{aligned}
\eeq
with $X_k(t_{-})=\eps_k^{-1}\left[x_{\eps_k}(t_-)-x_2\right]+x_2$. Furthermore, using the definition of $X_k(t)$ evaluated at $t=t_{-}$, followed by assumption \ref{item:4.2} of this lemma, yields, for all $k$ large enough,
\beq
\label{eq:X_k at t-}
X_k(t_-) = \ep_k^{-1}\left(x_{\ep_k}(t_-)-x_2\right)+x_2 > x_2.
\eeq 

We now seek to bound the right-hand side of \eqref{4.9} from below. To do this, we will use the explicit formula for $\partial_x v_{\eps_k}(X_{k}(t),t,\omega)$, analogous to \eqref{newBD2}, to express  $\partial_x v_{\eps_k}(X_{k}(t),t,\omega)$ in terms of $v_{\ep_k}$ and $a$. We'll use the inequality \eqref{eq:item:4.7} to bound $v_{\ep_k}$ from below. Then, we will apply Lemma \ref{L.2.2} to ``extract the average" of $a$ in a uniform way.

We return to \eqref{eq:vepk interior}, the equation that $v_{\ep_k}$ satisfies in the its positivity set. First, define $\lambda:=x_2-x_3$. Since assumption \ref{item:4.7} says $x_3<x_2$, we have $\lambda>0$. (Note that this agrees with the notation in the proof of Lemma \ref{lem for prop 1}, namely \eqref{6.16.5}.) Moreover, according to assumption \ref{item:4.7} of this lemma, we have $x_3\geq x_0-R$, which implies $\lambda\leq \tilde\lambda$. Thus, the first line of \eqref{eq:vepk interior} holds for $x\in (x_2-\lambda/\ep_k, X_k(t))$. Second, we use the fact that $a$ is stationary to replace $a(x-x_2+x_2/\eps_k,\omega)$ by $a(x,\tau_{y_k}\omega)$.
Thus, by \eqref{newBD2}, 
\begin{align*}
&|\partial_x v_{\eps_k}(X_{k}(t),t,\omega)|=v_{\eps_k}(x_2-{\lambda}/\eps_k,t,\omega)\left(\int_{x_2- {\lambda}/\eps_k}^{X_k(t)}\frac{1}{a(y,\tau_{y_k}\omega)dy}\right)^{-1}\frac{1}{a(X_{k}(t),\tau_{y_k}\omega)}\\
&\qquad=\frac{v_{\eps_k}(x_2- {\lambda}/\eps_k,t,\omega)}{X_k(t)-x_2- {\lambda}/\eps_k}(X_k(t)-x_2+ {\lambda}/\eps_k)\left(\int_{x_2- {\lambda}/\eps_k}^{X_k(t)}\frac{1}{a(y,\tau_{y_k}\omega)dy}\right)^{-1}\frac{1}{a(X_{k}(t),\tau_{y_k}\omega)}.
\end{align*}
We'll now bound $v_{\ep_k}(x_2- {\lambda}/\eps_k,t,\omega)$ from below. Indeed, 
evaluating  the inequality \eqref{eq:item:4.7} at time $(\ep_k(t-t_-)+t_-)\in (t_-, t_0)$ and multiplying by $\ep_k^{-1}$ yields,
\[
q_2(r_2(t-t_{-})+\eps_k^{-1}\lambda)+\eps^{-1}_kq_2\delta_2/2<\ep_k^{-1}p_{\ep_k}(x_2-\lambda, \ep_k(t-t_-)+t_-).
\]
Upon recalling the definition of $v_{\ep_k}$ we therefore find,
\[
q_2(r_2(t-t_{-})+\eps_k^{-1}\lambda)+\eps^{-1}_kq_2\delta_2/2< v_{\eps_k}(x_2-\eps_k^{-1}\lambda,t).
\]
Thus, using the previous inequality to bound the right-hand side of the explicit expression for $|\partial_xv_{\ep_k}(X_k(t), t)|$ from below, and denoting 
\[
{H}(t,S):=\frac{r_2(t-t_{-})+\eps_k^{-1}\lambda+\eps^{-1}_k\delta_2/2}{S-x_2+\eps_k^{-1}\lambda},
\]
we find, for $t\in \left(t_-+\frac{t_0-t_--R}{\ep_k}, t_-+\frac{t_0-t_-}{\ep_k}\right)$,
\[
|\partial_xv_{\ep_k}(X_k(t), t)|\geq H(t, X_k(t))(X_k(t)-x_2+ {\lambda}/\eps_k)\left(\int_{x_2- {\lambda}/\eps_k}^{X_k(t)}\frac{1}{a(y,\tau_{y_k}\omega)dy}\right)^{-1}\frac{1}{a(X_{k}(t),\tau_{y_k}\omega)}.
\]
Next, by Lemma \ref{L.2.2}, since $x_2$ is fixed and $|y_k|\leq C/\eps_k$, 
\[
\lim_{\eps_k\to0}\sup_{z\geq x_2}\left|\frac{1}{z-x_2+ {\lambda}/{\eps_k}}\int_{-z+x_2- {\lambda}/\eps_k}^{0}\frac{1}{a( y,\tau_{y_k+z}\omega)}dy-\frac1{\ol a}\right|=0\quad \text{\rm a.s. in }\omega.
\]
Let us take $z=X_k(t)>x_2$ for $t\geq t_-$. Hence,
there exist a deterministic constant  $c_1\in (0,1)$, which depends only on the assumptions, and $\xi:(0,1)\times\Sigma\to [c_1,1]$ (we use the notation $\xi_{k}(\omega)=\xi(\eps_k,\omega)$) such that  a.s. in $\omega$, $\lim_{k\to \infty}\xi_k(\omega)=1$ and uniformly for all $t\geq t_{-}$,
\[
(X_k(t)-x_2+ {\lambda}/\eps_k)\left(\int_{x_2- {\lambda}/\eps_k}^{X_k(t)}\frac{1}{a(y,\tau_{y_k}\omega)dy}\right)^{-1}\geq \ol a\, \xi_k(\omega).
\] 
Combining the previous inequality with the bound from below on  $|\partial_xv_{\ep_k}(X_k(t), t)|$
and recalling the notation introduced in \eqref{3.14}, we obtain
\begin{align*}
|\partial_x v_{\eps_k}(X_{k}(t),t,\omega)|&\geq {H}(t, X_k(t))\frac{q_2\ol a\,\xi_k(\omega)}{a(X_{k}(t),\tau_{y_k}\omega)}= {H}(t, X_k(t))q_2\xi_k(\omega)\mu(X_k(t),\tau_{y_k}\omega).    
\end{align*}
Thus, \eqref{4.9} and the previous line imply, for $t\in \left(t_-+\frac{t_0-t_--R}{\ep_k}, t_-+\frac{t_0-t_-}{\ep_k}\right)$,
\beq
\label{3.165}
\frac{d}{dt}X_k(t)\geq b(X_k(t),\tau_{y_k}\omega) {H}(t,X_k(t))q_2 \xi_k(\omega)  \mu(X_k(t),\tau_{y_k}\omega)+g(X_k(t),\tau_{y_k}\omega),
\eeq
which is the lower bound on $\frac{d}{dt}X_k$ that we had been seeking.

Now, we let $S_k(t,\tau_{y_k}\omega)$ be the unique solution to
\beq\lb{4.11}
\left\{
\begin{aligned}
&\frac{d}{dt} S_k(t,\tau_{y_k}\omega)= q_2 b(S_k(t,\tau_{y_k}\omega),\tau_{y_k}\omega) \mu(S_k(t,\tau_{y_k}\omega),\tau_{y_k}\omega)+g(S_k(t,\tau_{y_k}\omega))\\
&S_k(t_{-},\tau_{y_k}\omega)=x_2  .  
\end{aligned}
\right.
\eeq
Recalling the effective ODE \eqref{FB} implies,
\[
S_k(t+t_{-},\tau_{y_k}\omega)=S^{x_2}_{q_2}(t,\tau_{y_k}\omega).
\]
In addition, \eqref{eq:X_k at t-} and the definition of $S_k$ imply $X_k(t_-)>x_2=S_k(t_{-},\tau_{y_k}\omega)$. The goal is to show that $X_k(t)$ is not too much smaller than $ S_k(t,\tau_{y_k}\omega)$ when $k$ is large and for times $t$ greater than $t_-$ (to be specified precisely later).

To this end, we estimate
${H}(t,S_k(t,\tau_{y_k}\omega))$ from below.  
We will use the notation $r_{q_2}=\ol V(x_0,q_2)$. Further, we recall that $S_{q_2,\eps_k}^{x_2}$ was defined in \eqref{FBeps} via,
\[
S_{q_2,\eps_k}^{x_2}(t,\tau_{y_k}\omega)=\eps_k \left[S_{q_2}^{x_2}(\eps_k^{-1}t,\tau_{y_k}\omega)-x_2\right]+x_2.
\]
Let $C_{\max}$ be as in Lemma \ref{L.2.6}. Set $C_{\max,2}:=q_2C_{\max}+\|g\|_\infty$ and 
\[
\tau:=\min\left\{\frac{t_0-t_-}{2},\,\frac{\delta_2}{2r_{2}(t_0-t_{-})},\,\frac{\delta_2}{2C_{\max,2}}\right\}.
\]
Since $|y_k|\leq C/\eps_k$,  Proposition \ref{P.2.13} implies  that there exists a set of full measure $\Sigma_0\subset\Sigma$ such that, for all $\omega\in \Sigma_0$, 
\[
\limsup_{k\to \infty }\left[S^{x_2}_{q_2,\eps_k}(t,\tau_{y_k}\omega)-x_2\right]= r_{q_2} t
\]
uniformly for all $t\in [\tau,t_0-t_-]$. Let us take 
\beq
\label{eq:sigma}
\sigma:=(3r_2(t_0-t_{-}))^{-1}\delta_2>0.
\eeq
We have, for  $k=k(\tau,\sigma)$  sufficiently large, 
\beq\lb{3.15}
(1-\sigma)r_{q_2}(t-t_{-})\leq S_{k}(t,\tau_{y_k}\omega)  -x_2\leq (1+\tau)r_{q_2}(t-t_{-})
\eeq
whenever $t-t_{-}\in \eps^{-1}_k[\tau,t_0-t_- ]$. 

On the one hand, for $t-t_{-}\in \eps^{-1}_k[\tau,t_0-t_- ]$, we have, 
\begin{align*}
    {H}(t,S_{k}(t,\tau_{y_k}\omega))&\geq \frac{r_2(t-t_{-})+\eps_k^{-1}\lambda+\eps^{-1}_k\delta_2/2}{(1+\tau)r_{q_2}(t-t_{-})+\eps_k^{-1}\lambda} &\text{ by \eqref{3.15}}
    \\& \geq 
    \frac{r_2(t-t_{-})+\eps_k^{-1}\lambda+\eps^{-1}_k\delta_2/2}{r_{q_2}(t-t_{-})+\tau \ep_k^{-1}r_{q_2}(t_0-t_-)+\eps_k^{-1}\lambda} & \text{ since }t-t_{-}\in \eps^{-1}_k[\tau,t_0-t_- ]
    \\ & \geq
    \frac{r_2(t-t_{-})+\eps_k^{-1}\lambda+\eps^{-1}_k\delta_2/2}{r_{q_2}(t-t_{-})+\delta_2 \ep_k^{-1}r_{q_2}/(2r_2)+\eps_k^{-1}\lambda} &\text{ by our choice of $\tau$}
    \\&
    \geq r_2/r_{q_2}, &
\end{align*}
where the last inequality holds since $r_{q_2}>r_2$, which holds by \eqref{item:4.1'}. 

On the other hand, for $t-t_{-}\in [0, \eps^{-1}_k\tau]$, we use the bound $S_{k}(t,\tau_{y_k}\omega)\leq x_2+C_{\max,2}(t-t_{-})$, which holds by Lemma \ref{L.2.6}, followed by the definition of $\tau$, to find,
\[
{H}(t,S_{k}(t,\tau_{y_k}\omega))\geq \frac{r_2(t-t_{-})+\eps_k^{-1}\lambda+\eps^{-1}_k\delta_2/2}{C_{\max,2}(t-t_{-})+\eps_k^{-1}\lambda}\geq 1 \geq \frac{r_2}{r_{q_2}} .
\]
 Thus, we conclude
 \beq
 \label{eq:H}
{H}(t,S_{k}(t,\tau_{y_k}\omega))\geq  \frac{r_2}{r_{q_2}} \quad\text{ for }t-t_{-}\in \eps^{-1}_k[0,t_0-t_- ].
\eeq

We will now use these estimates on ${H}$ to show that a rescaled version of $S_k$ is a subsolution to the ODE that $X_k$ is a supersolution of: namely, \eqref{3.165}. To this end, consider 
\[
\tilde S(t):=S_k\left( \frac{r_2\xi_k(\omega)}{r_{q_2}}(t-t_{-})+t_{-},\tau_{y_k}\omega\right).
\]
Since $0\leq\frac{r_2\xi_k(\omega)}{r_{q_2}}\leq 1$, we have, for $t-t_{-}\in \eps^{-1}_k[0,t_0-t_- ]$,
\[
t_-\leq \frac{r_2\xi_k(\omega)}{r_{q_2}}(t-t_{-})+t_{-}\leq t.
\]
According to definition of $S_k$, we have $S_k(t_-)=x_2$, which therefore implies $\tilde{S}(t)\geq x_2$ for $t-t_{-}\in \eps^{-1}_k[0,t_0-t_- ]$. 
And, by the definition of $H$,  the map $t\mapsto H(t, S)$ is increasing in $t$ for $S>x_2-\ep_k^{-1}\lambda$. Together with \eqref{eq:H}, this implies,
for $t-t_{-}\in \eps^{-1}_k[0,t_0-t_- ]$,
\[
\frac{r_2}{r_{q_2}}\leq H\left(\frac{r_2\xi_k(\omega)}{r_{q_2}}(t-t_{-})+t_{-},S_{k}\left(\frac{r_2\xi_k(\omega)}{r_{q_2}}(t-t_{-})+t_{-},\tau_{y_k}\omega\right)\right)\leq H(t,\tilde S(t)).  
\]
Upon using \eqref{4.11}, followed by the above estimate on ${H}$ and the fact that $\frac{r_2\xi_k}{r_{q_2}}\leq 1$, we find, for $t-t_{-}\in \eps^{-1}_k[0,t_0-t_- ]$,
\begin{align*}
\frac{d}{dt}\tilde{S}(t) &= 
\frac{r_2\xi_k(\omega)}{r_{q_2}}q_2 b(\tilde{S}(t),\tau_{y_k}\omega) \mu(\tilde{S}(t),\tau_{y_k}\omega)+\frac{r_2\xi_k(\omega)}{r_{q_2}} g(\tilde{S}(t),\tau_{y_k}\omega)\\
&\leq {H}(t,\tilde{S}(t))\xi_k(\omega)q_2\,b(\tilde{S}(t),\tau_{y_k}\omega) \mu(\tilde{S}(t),\tau_{y_k}\omega)+g(\tilde{S}(t),\tau_{y_k}\omega).
\end{align*}
This is exactly the ODE that we've shown $X_k(t)$ to be a supersolution of: recall \eqref{3.165}. Further, let us note that $\tilde{S}(t_-)=x_2<X_k(t_-)$, where the second inequality is exactly from \eqref{eq:X_k at t-}. Therefore, the comparison principle for ODEs  yields
\beq\lb{4.8}
X_k(t)\geq S_k\left(\frac{r_2\xi_k(\omega)}{r_{q_2}}(t-t_{-})+t_{-},\tau_{y_k}\omega\right)
\eeq
whenever $t-t_{-}\in [0,\eps_k^{-1}(t_0-t_{-})]$.

Recall that $x_{\eps_k}$ denotes the free boundary of $p_{\eps_k}$. Using the definition of $X_k$, the bound \eqref{4.8} and the definition of $x_2$, followed by \eqref{3.15}, we find, 
\begin{align*}
x_{\eps_k}(t_0)&= \eps_k\left[X_k(\eps_k^{-1}(t_0-t_{-})+t_{-})-x_2\right]+x_2\\
    &\geq \eps_k\left[S_k(\eps_k^{-1}r_{q_2}^{-1}r_2\xi_k(\omega)( t_0-t_{-})+t_{-},{\tau_{y_k}}\omega)-x_2\right]-r_2(t_0-t_{-})+x_0+\delta_2\\
    &\geq ((1-\sigma)\xi_k(\omega)-1 )r_2(t_0-t_{-})+x_0+\delta_2
\\&=r_2(t_0-t_-)(\xi_k(\omega)-1 ) - r_2(t_0-t_{-})\sigma\zeta_k(\omega)+x_0+\delta_2.
\end{align*}
Recalling our choice of $\sigma$ in \eqref{eq:sigma} 
and using $\xi_k\in [c_1,1]$ yields that the middle term in the previous line is bounded from below by $-\delta_2/3$. Then, we take $k$ to be sufficiently large, depending only on $r_2,t_0-t_{-},\delta_2$ and $\omega$, such that
\[
(\xi_k(\omega)-1 )r_2(t_0-t_{-})\geq -\delta_2/3
\]
holds. 
Putting  these together, we obtain for all $\eps_k$ sufficiently small, depending on $\omega$ and the assumptions, that 
\[
x_{\eps_k}(t_0)\geq x_0+\delta_2/3
\]
holds, as desired.

\end{proof}

\subsection{Proof of main result}
\label{ss:proof of main result}

Finally, our main result follows from the proposition. Since the solutions are, in general, not continuous in time, the convergence result is almost everywhere in time.

\begin{proof}[Proof of Theorem \ref{thm:main}]
Let $p^0_\eps(x,\omega)$ be the unique solution to
\beq\lb{1.1'}
-\partial_x(A(x,\eps^{-1}x,\omega) \partial_x p_\eps)=F(x,\eps^{-1}x,\omega)
\eeq
in $\calO$ taking $0$ values outside $\calO$. By Lemma \ref{L.9.1}, we have that, almost surely in $\omega$,
\[
p^0_\eps(x,\omega)\to \ol p^0(x)
\]
where $\ol p^0$ is a continuous deterministic function uniquely solving 
\beq\lb{1.2'}
-\partial_x(\overline A(x) \partial_x \overline p(x))=\overline F(x)
\eeq
and $\ol p^0=0$ outside $\calO$.
Since $p_\eps(\cdot,0,\omega)$ satisfies the elliptic equation \eqref{1.1'} in $\Omega_{p_\eps}(0,\omega)$ and $F$ is strictly positive, we find that
\[
|\partial_xp_\eps(x,0,\omega)|=|\partial_xp^0_\eps(x,\omega)|
\]
is strictly positive for $x\in\partial\calO$, uniformly for all $\eps>0$ and $\omega$. This implies that the support of $p_\eps(x,0,\omega)$ expands immediately and uniformly in $\eps$ and $\omega$.

Let $\Omega_\eps:=\Omega_{p_\eps}$ and
\[
\Omega^*(t,\omega):=\limsup_{\eps\to 0,s\to t}\Omega_\eps(s)=\bigcap_{c_1,c_2\to 0}\bigcup_{|s-t|<c_1,\eps<c_2}\Omega_\eps(s),
\]
\[
\Omega_*(t,\omega):=\liminf_{\eps\to 0,s\to t}\Omega_\eps(s)=\bigcup_{c_1,c_2\to 0}\bigcap_{|s-t|<c_1,\eps<c_2}\Omega_\eps(s).
\]
We let $p^*(\cdot,t)$ solve
\[
-\partial_x(\overline A(x) \partial_x p)=\overline F(x)
\]
in the domain of $\Omega^*(t)$ with $0$ boundary data, and $p_*(\cdot,t)$ in the domain of $\Omega_*(t)$.

Since  $\Omega_\eps $ strictly expands at time $0$ uniformly for all $\eps$ and $\omega$, both $\Omega_*$ and $\Omega^*$ strictly expand at time $0$. We have for any $t>0$,
\[
\Omega_*(0,\omega)\subseteq \Omega_{* }(0,\omega)\subseteq \ol O\subset \Omega_{* }(t,\omega)\subseteq \Omega^* (t,\omega).
\]
Consequently,  for any $\delta>0$,
\beq\lb{4.21}
\Omega^* (0)\subset (1+\delta)\Omega_* (\delta,\omega).
\eeq

By Proposition \ref{P.3.2}, $\Omega^* $ is a viscosity subflow and $\Omega_* $ is a viscosity superflow of \eqref{1.3}. 
Since the coefficients in \eqref{1.3} are independent of time, one can check that that 
\[
(1+\delta )p_* (x,(1+\delta)t+\delta,\omega)
\]
is a viscosity supersolution to \eqref{1.3}.
Thus, it can also be verified directly by the definition of viscosity superflow that
\[
(1+\delta )\Omega_* ((1+\delta)t+\delta,\omega)
\]
is a superflow of \eqref{1.3}. 
Therefore, using \eqref{4.21}, the comparison principle (Theorem \ref{L.cp}) yields for all $t\geq 0$,
\[
\Omega^* (t,\omega)\subseteq (1+\delta )\Omega_* ((1+\delta)t+\delta,\omega).
\]
After passing $\delta\to 0$, this implies that
\[
p^* (x,t,\omega)\leq \sup_{s\to t}p_* (x,s,\omega).
\]

On the other hand, since $\Omega_{* }(t,\omega)\subseteq \Omega^* (t,\omega)$, the definitions of $p^* $ and $p_* $ yield
\[
p^* (x,t,\omega)\geq p_* (x,t,\omega)
\]
and so, in view of \eqref{half}, we actually get equality 
\[
p^* (x,t,\omega)^\#=p_*(x,t,\omega)^\#. 
\]
By Theorem \ref{T.3.4} and Lemma \ref{L.4.8}, there is a unique solution (viscosity flow) to \eqref{1.3}. In view of Definition \ref{D.1.4}, it follows that 
\[
(\Omega^*(\cdot,\omega) )^\#=(\Omega_*(\cdot,\omega))^\#=\ol\Omega(\cdot)
\]
is the unique viscosity flow of \eqref{1.3} with initial domain $\calO$. Since \eqref{1.3} and $\calO$ are deterministic, the flow is deterministic. Therefore
\beq\lb{4.13}
\begin{aligned}
p_* (x,t,\omega)^\#=p^* (x,t,\omega)^\#=\ol p(x,t)
\end{aligned}
\eeq
where $\ol p$ is the unique deterministic viscosity solution to \eqref{1.3} with initial data $\ol p^0$.

Next, to see the convergence of $p_\eps$ for almost every time, due to \eqref{4.13}, it suffices to show that $\ol p(x,t)$ is continuous in space-time for almost all $t$. Note that  $\ol p(\cdot,t)$ solves \eqref{1.2'} 
in $\{\ol p(\cdot,t)>0\}$ for all $t>0$, and consequently, $\ol p(x,t)$ is uniformly continuous in space. 
Since $\ol p$ is lower semicontinuous, $\{\ol p(\cdot,t)>0\}$ is an open subset of $\bbR$. Since the support of $\ol p$ is increasing continuously in Hausdorff distance, and  $\ol p(\cdot,t)$ solves an elliptic equation in the open set,
$\ol p$ is continuous at time $t$ as long as the open intervals consisting $\{\ol p(\cdot,t)>0\}$ does not merge at the time $t$. 
Since there are initially at most countably many open intervals in $\calO$, as the set expands in time, the open intervals consisting $\{\ol p(\cdot,t)>0\}$ only merge for countably many times. Therefore $\ol p(x,\cdot)$ is continuous for almost all times for each $x\in\bbR$. This finishes the proof.

Finally, we note that $\overline{A}$, $\overline{F}$, and $\overline{V}$ and  satisfy Assumption \ref{assump2} due to Lemmas \ref{L.9.1} and \ref{L.4.8}, respectively.
\end{proof}


\bibliographystyle{siam}

\end{document}